\documentclass{article}

\author{author}
\date{\today}
\title{title}

\usepackage[utf8]{inputenc}
\usepackage[T1]{fontenc}
\usepackage[english]{babel}
\usepackage{amsmath,amssymb}
\usepackage{bbm}
\usepackage{amsthm}
\usepackage{amssymb} 
\usepackage[mathscr]{eucal} 
\usepackage[matrix,arrow,tips,curve,ps]{xy}
\usepackage{amsmath} 
\usepackage{amscd}
\usepackage{xcolor}
\usepackage{listings}

\usepackage{mathrsfs}
\usepackage{floatflt,epsfig}
\usepackage{comment}
\usepackage{hhline}

\usepackage{enumitem}
\usepackage[a4paper,top=3cm,bottom=3cm,left=3.5cm,right=3.5cm,%
bindingoffset=5mm]{geometry}
\usepackage{url}
\usepackage{hyperref}

\theoremstyle{plain} 
\newtheorem{thm}{Theorem}[section] 
\newtheorem*{thm*}{Theorem}
\newtheorem{cor}[thm]{Corollary} 
\newtheorem{lem}[thm]{Lemma} 
\newtheorem{prop}[thm]{Proposition} 
\theoremstyle{definition} 
\newtheorem{defn}{Definition}
\newtheorem*{qst*}{Question}
\theoremstyle{remark} 
\newtheorem{rmk}{Remark} 
\theoremstyle{remark}

\newcommand{\D}[1]{{\mathbb#1}}

\newcommand{\C}{{\mathcal{C}}}

\newcommand{\Pw}{{\D{P}(\bigwedge^2 W^*)}}

\newcommand{\Pd}{{\D{P}(\Delta)}}

\newcommand{\git}{\mathbin{
  \mathchoice{/\mkern-6mu/}
    {/\mkern-6mu/}
    {/\mkern-5mu/}
    {/\mkern-5mu/}}}

\newcommand*{\sheafhom}{\mathscr{H}\kern -.5pt om}
\newcommand*{\sheafext}{\mathscr{E}\kern -.5pt xt}
\newcommand{\F}{{\mathcal{F}}}
\newcommand{\OOO}{\mathcal{O}_{\D{P}^4}}
\newcommand{\mapstoo}[1]{\:
	\xymatrix@1{\ar@{|->} [r] ^-{#1}&}\:}
\newlist{myenumerate}{enumerate}{1}
\setlist[myenumerate,1]{label=(\roman*)}

\newcommand{\DM}[1]{\textcolor{black}{#1}}

\renewcommand{\restriction}{|}

\renewcommand\emptyset{\varnothing}

\newcommand{\codim}{\textup{codim\,}}

\newcommand{\Gr}{\textup{Gr}}

\newcommand{\Hom}{\textup{Hom\,}}

\newcommand{\Pf}{\textup{Pf\,}}

\newcommand{\rk}{\textup{rk\,}}
\newcommand{\Sing}{\textup{Sing\,}}

\newcommand{\Supp}{\textup{Supp\,}}
\newcommand{\coker}{\textup{coker\,}}

\newcommand{\Stab}{\textup{Stab}}
\newcommand{\im}{\textup{Im}}
\newcommand{\Pvw}{{\D{P}({V_{n+1}}^*\otimes \bigwedge^2 W^*)}}
\newcommand{\Pa}{{\D{P}(A)}}
\newcommand{\Pb}{{\D{P}(B)}}

\AtEndDocument{\bigskip{\footnotesize%
		\textsc{Gaia Comaschi} \par
		\textsc{Departamento de Matemática, IMECC - UNICAMP} \par  
		
}}

\begin{document}
	\title{\textbf{Stable linear systems of skew-symmetric forms of generic rank $\le 4$}}
	\author{Gaia Comaschi}

	\date{}

\maketitle

\begin{abstract}
	\noindent
Given a 6-dimensional complex vector space $W$, we consider linear systems of skew-symmetric forms on $W$. 
The $n$-dimensional linear systems this kind, that can also be interpreted as $n$-dimensional linear subspaces of $\D{P}(\bigwedge^2 W^*)$, are parametrized by the projective space $\D{P}(\D{C}^{n+1}\otimes \bigwedge ^2 W^*)$. 
We analyze the $SL(W)$ action on this projective space and the GIT stability of linear systems with respect to this action.
We present a classification of all stable orbits of linear systems whose generic element is a tensor of rank 4. 
\end{abstract}
\section{Introduction}
Given a complex vector space $W$ of dimension 6, we study linear systems of skew-symmetric forms on $W$. Linear systems of this kind 
may also be considered as linear subspaces of $\D{P}(\bigwedge^2 W^*)$ or, yet again,
as $6\times 6$ skew-symmetric matrices of linear forms.  Given now a $(n+1)$-dimensional complex vector space $V_{n+1}$, we notice that a $n$-dimensional linear system of alternating forms on $W$ defines a linear embedding $\D{P}(V_{n+1})\to \Pw$ providing a point in the projective space
$\D{P}(V_{n+1}^*\otimes \bigwedge^2W^*)$. The group $SL(W)$ acts on $\D{P}(V_{n+1}^*\otimes \bigwedge^2 W^*)$ and it is then rather natural to ask the question whether it is possible to classify all the orbits and if this is the case, to actually present such a classification.

The classification of linear spaces of forms is a subject that has been considered in various contexts. 
The first results already appeared in the classical work of Weierstrass \cite{We} and Kronecker \cite{Kr}; a more modern formulation of the subject was later presented by Sylvester \cite{Syl} using the language of vector bundles on projective spaces. This approach has been recently adopted in a series of works by Boralevi, Faenzi et al. 
(\cite{BFM}, \cite{BFL}, \cite{BM}.) 

It is easy to classify pencils of alternating matrices of size 6 but things reveal to be much more complicated already in dimension 2 (see \cite{MM} for a classification of planes of tensors having constant rank equal to 4). For this reason we might then start by restricting only to certain orbits, those that are \textit{stable} (in the sense of Mumford's geometric invariant theory (GIT)).
In the present paper we give a complete classification of stable orbits of linear systems whose generic element is a tensor of rank 4. Our methods mainly rely on the study of the geometry of the corresponding linear subspaces of $\Pw$, namely of linear subspaces of the Pfaffian hypersurface $\Pf\subset \Pw$.

Here is the plan of the paper. After recalling some basics on the geometry 
of the Grassmanniann $\Gr(2,W^*)$ of lines in $\D{P}(W^*)\simeq \D{P}^5$, and of the Pfaffian hypersurface $\Pf$, we introduce $n$-dimensional linear systems of alternating forms 
focusing in particular on those of generic rank $	\le 4$.

Section 2 is devoted to the study of GIT stability. 
Adapting the method used by Wall \cite{Wall}, we formulate a criterion for (semi)stability and we show that
this criterion allows us to characterize the instability and the non-stability of a linear space $\Pa \subset \Pw$ by the existence of certain subspaces of its orthogonal $\Pa^{\perp}\subset \D{P}(\bigwedge^2 W)$ (this ``geometric formulation'' of stability will be the one we will mainly use). 

In section 3 we present some necessary condition for the stability of a linear space \mbox{$\D{P}(A) \subset \Pf$}, proving:
\begin{thm*}[Theorem \ref{dim_le3}]Let $\Pa\subset \Pf$ be a $n$-dimensional stable linear space. Then the following hold:
	\begin{itemize}
		\item $n\ge 3$;
		\item a general plane $\Pb \subset \Pa$ is such that $\Pb \cap \Gr(2,W^*)=\emptyset$;
		\item any plane $\Pb \subset \Pa$ such that $\Pb \cap \Gr(2,W^*)=\emptyset$ is $SL(W)$-equivalent to the plane $\pi_g$:
		$$\pi_g=\langle e_1\wedge e_4 + e_2\wedge e_5, \ e_1 \wedge e_6+ e_3\wedge e_5, \ e_2\wedge e_6 - e_3\wedge e_4 \rangle$$
		for a basis $e_1,\ldots e_6$ of $W^*$;
		\item $n\le 5$.
	\end{itemize} 
	\end{thm*}

The starting point for the proof of the theorem is the following. Since we are dealing with subspaces $\Pa$ of the Pfaffian hypersurface $\Pf$, it is rather natural to ask if ever these spaces meet the Grassmannian $\Gr(2,W^*)$ and consequently to detect how these intersections might affect the stability of the entire $\Pa$.
First of all we prove that if $\Pa \subset \Pf$ is stable and if ever there exists a tensor $\omega \in \Pa$ of rank 2, the locus $\Pa \cap \Gr(2,W^*)$ must have codimension 3.
We deduce then that if $\Pa$ has dimension $n\le 2$, it must have constant rank equal to 4; on the contrary, if its dimension is at least 3, $\Pa$ must contain a plane of tensors that all have rank 4.
Both cases can be dealt with by studying planes of alternating matrices of constant rank 4. To this aim we use a result proved by Manivel-Mezzetti in \cite{MM} that states the existence of only four $SL(W)$ orbits of planes of such kind, and we analyze their stability. 
We initially show that none of these orbits is stable and that this prohibits the stability of any linear system $\Pa \subset \Pf$ of dimension less than or equal to 2.
Accordingly a stable $\Pa \subset \Pf$ has dimension at least 3; we prove then that $\Pa$ always contains a plane $\Pb$ of constant rank 4 and that moreover the stability assumption implies that $\Pb$ may only belong to one of the four $SL(W)$-orbits, namely that of $\pi_g$.
These preliminary results are the main ingredients 
of our classification. They imply in particular that a stable $\Pa \simeq \D{P}^n, \ \Pa\subset \Pf$ can always be represented in the form:
$$\Pa=\langle \Pb, \omega_3, \ldots 	,\omega_n\rangle$$
with $\Pb \in SL(W)\cdot \pi_g$ and with the $\omega_i$s, $i=3,\ldots,n$ belonging to a $(n-3)$-dimensional component of $\Pa \cap \Gr(2,W^*)$.  
From the hypothesis $\Pa \subset \Pf$, we deduce that such an irreducible component must be contained in a 3-dimensional scroll $Z\simeq S_{(2,2,2)}$ carrying the structure of a conic bundle over $\D{P}(B)$, $Z\xrightarrow{\pi} \Pb$. 
The study of the subvarieties of $Z$ will allow us to conclude that there exist no stable linear systems of generic $\le 4$ and of dimension $> 5$, and this will enable us to complete our classification.

In the subsequent sections we present the actual classification of stable orbits. 
We find only one stable orbit in dimension three:
\begin{thm}\label{dim_3}
		Let $\Pa \subset \Pf$ be a stable three dimensional linear space. Then $\Pa$ is $SL(W)$-equivalent to the space:
		$$\langle e_1\wedge e_4 + e_2\wedge e_5, e_1 \wedge e_6 + e_3\wedge e_5, e_2 \wedge e_6 - e_3 \wedge e_4, (e_1 -e_5) \wedge (e_2 +e_4)\rangle$$
 and $\Pa \cap \Gr(2,W^*)$ consists of two distinct points. 
	\end{thm}
A stable 3-dimensional linear space $\Pa$ must indeed be of the form $\Pa=\langle \Pb, \omega_3\rangle$ with \mbox{$\Pb\in SL(W)\cdot \pi_g$} and $\omega_3$ a point lying in $Z$; from the construction of $Z$, we know that $\omega_3$ belongs to a unique fiber of $\pi$, so that there exists a unique $\omega \in \Pb$ for which we have $\omega_3 \in \pi^{-1}(\omega)$.
We show that $\Pa$ will be stable whenever $\omega\not\in \D{T}_{\omega_3}\Gr(2,W^*)$.

The classification of stable 3-planes will help us to treat the higher dimensional cases since a $n$-dimensional $\Pa \subset \Pf$, $n=4, 5$ will be stable whenever it contains a stable linear subspace of dimension 3.
We prove that there exist  two stable orbits of 4-dimensional linear systems of generic rank 4:

 \begin{thm}\label{dim_4}
			Let $\Pa$ be a stable 4-dimensional linear system of \DM{skew-symmetric form}s of generic rank less then or equal to four. Then $\Pa$ is $SL(W)$\DM{-equivalent} to one of the following two spaces:
		$$ \DM{\langle}e_1\wedge e_4 + e_2\wedge e_5, \ e_1 \wedge e_6+ e_3\wedge e_5, \ e_2\wedge e_6 - e_3\wedge e_4, \ e_1\wedge e_2, \ e_4\wedge e_5\DM{\rangle}$$
		or 
		$$ \DM{\langle}e_1\wedge e_4 + e_2\wedge e_5, \ e_1 \wedge e_6+ e_3\wedge e_5, \ e_2\wedge e_6 - e_3\wedge e_4, \ (e_1-e_5)\wedge (e_2+e_4), \ (e_1-e_5)\wedge (e_3+e_6)\DM{\rangle}.$$
		Furthermore in the first case $\Pa$ meets the Grassmannian $\DM{\Gr}(2,W^*)$ along a smooth conic isomorphic to $\D{P}^2\cap \DM{\Gr}(2,4)$; in the second case $\Pa$ intersects $\DM{\Gr}(2,W^*)$ along a pair of disjoint lines.  
		\end{thm}

Finally, passing to the 5-dimensional case, we show that because of the inclusion \mbox{$\Pa \subset \Pf$}, every hyperplane in $\Pa$ can belong only to one among the two stable orbits of 4-planes; 
this leads to the existence of just one stable orbit: 
\begin{thm} Let $\Pa \subset 	\Pf$ be a stable linear space of dimension 5. Then $\Pa$ is $SL(W)$ equivalent to the space:
	$$ \DM{\langle}e_1\wedge e_4 + e_2\wedge e_5, \ e_1 \wedge e_6+ e_3\wedge e_5, \ e_2\wedge e_6 - e_3\wedge e_4, \ (e_1-e_5)\wedge (e_2+ e_4), (e_1-e_5)\wedge (e_3+e_6), (e_2+e_4)\wedge (e_3+e_6)\DM{\rangle}$$
	and $\Pa$ will meet the Grassmannian $\Gr(2,W^*)$ along a pair of disjoint planes.
	
	\end{thm}

\section{Preliminaries}
\subsection{Basics on the geometry of $\Gr(2,W^*)$ and $\Pf$}
		In this section we recollect a few generalities about the Grassmannian and the Pfaffian hypersurface. 
		 Let $W$ be a complex vector space of dimension 6 and consider $\bigwedge^2 W^*$, the 15-dimensional vector space of skew-symmetric bilinear forms on $W$.
		Each element $\omega \in \bigwedge^2 W^*$ determines a linear morphism from $W$ to $W^*$ representable by a $6 \times 6$ skew-symmetric matrix $M_{\omega}$.  We define the rank of $\omega$ as the rank of $M_{\omega}$. Note that, as $\omega$ is a skew-symmetric bilinear form, its rank is always even, so \mbox{$\forall \ \omega \in \bigwedge^2 W^*$}, \  $ rk(\omega)=2k$, \mbox{$ 0\le k\le 3 $.} 
		Consider now $\Pw\simeq \D{P}^{14}$, the projective space of lines in $\bigwedge^2 W^*$. 
		Since the rank and the kernel of any tensor in $\bigwedge^2 W^*$ are invariant under multiplication by a non-zero scalar, the notions of rank and kernel are well-defined also for elements in $\Pw$ (for $\omega \in \Pw$, $\ker(\omega)<W$ is thus defined as the kernel of any tensor in the line corresponding to $\omega$).  

		The locus of all points having rank less than or equal to 4 defines a hypersurface, named the Pfaffian hypersurface, $\Pf \subset \Pw$: 
		$$\Pf:=\{\omega \in \Pw \mid rk (\omega) \le 4\}.$$
		
		Because of the fact that  an arbitrary element $\omega$ in $\Pw$  is representable by a skew-symmetric matrix, $det \ \omega $ is a square: $det \ \omega=\Pf(\omega)^2$. 
		Thus the Pfaffian, being determined by the equation $\Pf(\omega)=0$, is a cubic hypersurface (as $\dim(W^*)=6$ ). 
		Consider now $\Gr(2,W^*)$, the Grassmannian of lines in $\D{P}(W^*)\simeq \D{P}^5$. By means of the Plucker's embedding, we can realize $\Gr(2,W^*)$ as a smooth subvariety of $\Pw$ of dimension 8 and degree 14 (more generally the Grassannian $\Gr(2, n+1)$ of lines in $\D{P}^n$, is a smooth subvariety of $\D{P}(\bigwedge^2 \D{C}^{n+1})$ of dimension $2(n-1)$ and degree $\frac{[2(n-1)]!}{(n-1)!n!}$.)
		As any $\omega \in \Pw$ having rank equal to 2 might be written as an $\textit{indecomposable}$ tensor of the form $\omega= v_1 \wedge v_2, \ v_1,v_2\in W^*$,  we see that $\Gr(2,W^*)$ coincides with the locus:
		$$\Gr(2,W^*)=\{\omega \in \Pw \mid rk(\omega)=2\}$$ 
		We can therefore identify $\Gr(2,W^*)$ with $\Sing(\Pf)$ (the set of singular points in $\Pf$).
		Accordingly, if we consider the projective tangent space $\D{T}_{\omega}\Pf$ to $\Pf$ at $\omega$, if ever \mbox{$\omega \in \Gr(2,W^*)$}, then $\D{T}_{\omega}\Pf=\Pw$; if $\omega \in \Pf\setminus \Gr(2,W^*)$ instead, $\D{T}_{\omega}\Pf$ coincides with the hyperplane defined by $\ker(\omega) \in \Gr(2,W)$ (as $\omega \in \Pf \setminus \Gr(2,W^*)$ defines a linear morphism $W\rightarrow W^*$ of rank 4, its kernel $\ker(\omega)$ determines a point in the dual Grassmannian $\Gr(2,W)$).
		
		These arguments show that the Gauss map $\gamma$ defined by $\Pf$:
		\begin{align*}
		\gamma :\D{P}(\bigwedge ^2 W^*)&\dashrightarrow \D{P}(\bigwedge ^2 W)\\
		                    \omega & \mapsto d_{\omega}\Pf
		\end{align*}
		is a rational map of degree two whose indeterminacy locus is $\Gr(2,W^*)=\Sing(\Pf)$ and mapping $\Pf\setminus \Gr(2,W^*)$ to $\Gr(2,W)$. 
		$\gamma$ is determined by a linear system of quadrics $\mathcal{Q}\subset S^2(\bigwedge^2 W), |\mathcal{Q}|\simeq \D{P}(\bigwedge^4 W)\simeq \D{P}(\bigwedge^2 W^*)$, referred to as \textit{Plucker's quadrics}, whose base locus is thus $\Gr(2,W^*)$.

	    If now we look at the projective tangent spaces to the Grassmannian, we have that for each point $\omega \in \Gr(2,W^*)$, $\D{T}_{\omega}\Gr(2,W^*)\simeq \D{P}^8$;
	    denoting by $L_{\omega}\subset W^*$ the 2-dimensional subspace of $W^*$ corresponding to $\omega$, $\D{T}_{\omega}\Gr(2,W^*)$ coincides with the locus of tensors in $\Pw$ whose restriction to $L_{\omega}^{\perp}$ is identically equal to zero (implying in particular that $\D{T}_{\omega}\Gr(2,W^*)\subset \Pf$).
	      This 8 plane is spanned by points $\omega' \in \Gr(2, W^*)$ corresponding to subspaces $L_{\omega}'$ of $W^*$ meeting $L_{\omega}$
	    (equivalently to lines in $\D{P}(W^*)$ meeting $\D{P}(L_{\omega})$) and it is cut out by hyperplanes lying in \mbox{$\D{P}(\bigwedge^2 L_{\omega}^{\perp})\subset \D{P}(\bigwedge^2 W)$},  $\ \D{P}(\bigwedge^2 L_{\omega}^{\perp})\simeq \D{P}^5$ (this latter is the linear span of the 4-dimensional quadric $\Gr(2, L_{\omega}^{\perp})\simeq \Gr(2,4)$).
	    Once we have described the tangent spaces to the Grassmannian, we notice that we can identify the fiber $\gamma^{-1}(\theta)$ of $\gamma$ over a point $\theta \in \Gr(2,W)$ with $\D{P}(\bigwedge^2 L_{\theta}^{\perp})$ ($L_\theta$ being the 2-dimensional subspace of $W$ corresponding to $\theta$), the five dimensional linear space of hyperplanes containing $\D{T}_{\theta}\Gr(2,W)$. 
		
		\subsection{Linear systems of skew-symmetric forms}\label{sect_hw}
Consider $V_{n+1}, \ W$, a pair of complex vector spaces of respective dimensions $n+1$ and $6$, and $\phi: V_{n+1}\rightarrow \bigwedge^2 W^*$ an injective linear map. Defining such a morphism \mbox{$\phi \in Hom_{\D{C}}(V_{n+1},\bigwedge^2 W^*)$} is clearly equivalent to defining a ($n+1$)-dimensional linear subspace $A:=\phi(V_{n+1})$ of $\bigwedge^2 W^*$ or, yet again, $6\times 6$ skew-symmetric matrix $M_{A}$ whose entries are elements of ${V_{n+1}}^*$. 
Choose indeed a basis $Y_0, \ldots Y_n$ of $V_{n+1}$ and a dual basis $X_0,\ldots X_n$ of $V_{n+1}^*$. Each tensor \mbox{$\omega_k \in A, \ \omega_k:=\phi(Y_k), \ k=0,\ldots n,$} defines a skew-symmetric matrix $M_k=(m_{ij}^k)$ of size 6; the matrix $M_{A}$ corresponding to $\phi$ is therefore the matrix whose $(ij)$-th entry is the linear form:
$${(M_{A})}_{ij}=\sum_{k=0}^n m_{ij}^k X_k.$$

\vspace{2mm}
 
We consider now the $n$-dimensional projective space $\D{P}(A)\subset \Pw$; this is the \mbox{$n$-dimensional} linear system of skew-symmetric forms generated by the tensors $\omega_0,\ldots, \omega_n$. 
The intersection of $\Pa$ with $\Pf$ locates tensors in $\Pa$ having rank equal to 4. For $\Pa$ general, this intersection defines a cubic hypersurfaces in $\Pa$ whose equation is \mbox{$\Pf(M_{A})=0$}. 
If ever $\Pf(M_{A})\equiv 0$ instead, $\Pa$ will be entirely contained in $\Pf$.
Tensors in $\Pa$ having rank 2 coincide with points belonging to $\Pa\cap \Gr(2,W^*)$.
The locus $\Pa\cap \Gr(2,W^*)$ is a closed subvariety of $\Pa$ defined by an intersection of quadrics. Specifically, it is the base locus of $|\mathcal{Q}\restriction_{\Pa}|$ (the linear system of quadrics in $\mathcal{Q}$, restricted to $\Pa$), so that: 
$$\Pa\cap \Gr(2,W^*)=\Pa\cap (\bigcap_{Q \in \mathcal{Q}} Q).$$

The equations defining $\Pa\cap \Gr(2,W^*)$ can be written down explicitly taking the $4\times 4$ minors of $M_{A}$.

	\paragraph{Linear systems of generic rank 4}
	
	\
	
	In this article we will be mainly interested in the study of linear systems of skew-symmetric forms whose generic element is a tensor of rank 4. From what we have just illustrated this happens if and only if the corresponding linear space $\Pa$ is entirely contained in $\Pf$, hence our task reduces to the study of linear subspaces of the Pfaffian hypersurface. Now, given a $n$-dimensional linear space $\Pa\subset \Pw$ and a point $\omega \in \Pa \cap \Pf$,
	the locus of points $\omega' \in \Pw$ such that $\overline{\omega'\omega}\subset \Pf$, is a complete intersection sextic $D_{\omega}$, defined by :
	\begin{equation}\label{sextic}
	D_{\omega}:= \D{T}_{\omega}\Pf \cap Q_{\omega} \cap \Pf
	\end{equation}
	where $Q_{\omega}$ is a quadric hypersurface, more precisely a quadric cone with vertex in $\omega$, belonging to $\mathcal{Q}$.
	If furthermore $\Pa\subset \Pf$, then any hyperplane of $\Pa$ not passing through $\omega$ must be contained in $D_{\omega}$.
	\subsection*{Notation}
	We will always assume that our base field is $\D{C}$. Throughout the article we will adopt the following notation:
	\begin{itemize}
	\item $W\simeq \D{C}^6$ is a complex vector space of dimension 6, $\langle e_1,\ldots e_6\rangle$ is a basis of $W^*$ and $\langle u_1,\ldots, u_6\rangle$ is a basis of $W$ dual to  $\langle e_1,\ldots e_6\rangle$;
	\item $V_{n+1}\simeq \D{C}^{n+1}$ is a complex vector space of dimension $n+1$;
	\item for $\omega \in \Gr(2,W^*)$, $L_{\omega}\subset W^*$ is the corresponding two-dimensional subspace of $W^*$.
	\item for $\omega \in \Pf$, $Q_{\omega}$ and $D_{\omega}$ are, respectively, the quadric cone and the sextic defined in equation \ref{sextic};
	
	\end{itemize}
	According to the characterizations of linear systems of skew-symmetric forms presented at the beginning of the section, in what follows we will identify a linear system with the corresponding subspace of $\Pw$ (so that in the upcoming sections we will simply talk about ``GIT stability of linear spaces of forms''). 
\section{GIT stability}
At the beginning of the previous section, 
	we observed that given a $n$-dimensional linear system \mbox{$\Pa\subset \Pw$}, the vector space $A\subset \bigwedge^2 W^*$ corresponds to an element in \mbox{$\Hom_{\D{C}}(V_{n+1}, \bigwedge^2 W^*)\simeq {V_{n+1}}^*\otimes \bigwedge^2 W^*$} (so that $\Pa$ defines a point in the projective space $\D{P}({V_{n+1}}^*\otimes \bigwedge^2 W^*)\simeq \D{P}^{15(n+1)-1}$).
          We have a natural action of the group \mbox{$GL(n+1,\D{C})\times GL(W)$} on the vector space $V_{n+1}^*\otimes \bigwedge ^2 W^*$: 
          given $A\subset \bigwedge^2 W^*, \ A\simeq \D{C}^{n+1}$, $GL(n+1, \D{C})$ simply acts as a change of coordinates on $A$, so that if $A=\langle \omega_0, \ldots \omega_n\rangle$ and $h=(h_{ij}), \ h\in GL(n+1, \D{C})$,
          then 
          $$h\cdot A=	\langle \omega_0', \ldots , \omega_n'\rangle, \ \omega_i'=\sum _{k=0}^n h_{ki}\omega_k$$ (equivalently if $\phi:V_{n+1} \to \bigwedge^2 W^*$ is the injective linear map corresponding to $A$, $h\cdot \phi$ is the morphism defined by $h\cdot \phi(v)=\phi(h\cdot v)$).
          Nevertheless we will be mainly interested in the $GL(W)$ action that is described as follows: for 
          $A=\langle \omega_0, \ldots \omega_n\rangle$ a ($n+1$)-dimensional subspace of $\bigwedge^2 W^*$ and $g\in GL(W)$, we have:
	$$g\cdot A=\langle g\cdot \omega_0, \ldots, g\cdot \omega_n\rangle, \hspace{5mm} $$
	(that is, for $\phi: V_{n+1}\to \bigwedge^2 W^*$ the corresponding linear map, $(g\cdot \phi)(v)=g\cdot (\phi(v))$). 
	Recall that given $\omega\in \bigwedge^2 W^*, g \in GL(W), \ g\cdot \omega$ is the \DM{skew-symmetric form} \DM{defined} by $g\cdot \omega(u,u')=\omega(g\cdot u, g\cdot `)$, for any pair $u,\ u'$ of vectors in $W$. Therefore, \DM{if $M_{\omega}, \ M_{g\cdot \omega}$ denote} the skew-symmetric matrices representing $\omega$ and $g\cdot \omega$ \DM{respectively }, we have \mbox{$M_{g\cdot \omega}=g M_{\omega}\  g^T $}. Similarly, given $A \subset \bigwedge^2 W^*, \ A \simeq \D{C}^{n+1}$, if $M_{A}, \ M_{g\cdot A}$ denote the skew-symmetric matrices of linear forms representing $A$ and $g\cdot A$ respectively, then \mbox{$M_{g\cdot A}=g M_{A}\  g^T $}.

	The action of $GL(W)$ on $V_{n+1}^*\otimes 	\bigwedge^2 W^*$ induces an action on its projective space defined by:
	$$ g\cdot \D{P}(A):=\D{P}(g\cdot A);$$
	moreover since for any non-zero scalar $t \in \D{C}^*$, $\D{P}(tg\cdot A)=\D{P}( g\cdot A)$, we see that $GL(W)$ acts on $\D{P}({V_{n+1}}^*\otimes \bigwedge^2 W^*)$ as $PGL(W)$. 
	Under these circumstances  it is thus possible to apply geometric invariant theory (GIT for short) to get a notion of (semi)stability for linear systems of skew-symmetric forms. 
	
	\subsection{The GIT criterion}
	In order to apply GIT it is first convenient to restrict to the action of $SL(W)$. We recall here briefly how GIT stability (with respect to the $SL(W)$ action) for points in $\D{P}({V_{n+1}}^*\otimes \bigwedge^2 W^*)$ is defined. (We refer to \cite{Dol}, \cite{GIT}, for a more detailed and exhaustive dissertation on Geometric Invariant Theory).
	\begin{defn}Let $\Pa\subset \Pw$ be a $n$-dimensional linear space. 
	\begin{itemize}
	\item $\Pa$ is semistable if and only if $0\notin SL(W)\cdot A$;
	\item $\Pa$ is stable if and only if $SL(W)\cdot A$ is closed and the stabilizer $\Stab(A)$ of $A$ is finite.
	\end{itemize} 
	\end{defn}  
	Nevertheless it is possible to detect the stability of points in $\Pvw$ just by studying their stability with respect to the action of certain subgroups of $SL(W)$, namely 1 parameter subgroups ($1$-PS for short).
	\begin{defn}A 1-PS subgroup $\lambda$ of $SL(W)$ is a morphism of algebraic groups:
	$$\lambda: \D{G}_m \rightarrow SL(W).$$
	\end{defn}
	
	\
	
	Hilbert and Mumford established indeed a criterion (see \cite{GIT}) stating the following:
	
	\begin{thm*}[\textbf{Hilbert\DM{-Mumford} criterion\DM.}] $\Pa\subset \Pw$ is (semi)stable if and only if it is $\lambda$-(semi)stable (i.e. (semi)stable with respect to the action of $\lambda(\D{G}_m)$) for any 1-parameter subgroup $\lambda$ of $SL(W)$.  
	\end{thm*}
	\
	
	Therefore given $\Pa \simeq \D{P}^n, \ \Pa \subset \Pw$, 
we have that $\Pa$ fails to be \DM{semistable} if and only if there exists \DM{a 1-PS $\lambda$} of $SL(W)$ such that $0=\lim_{t\to \infty} \lambda(t)\cdot A$ (so that  $0 \in \overline{\lambda(\D{G}_m)\cdot A}$).  $\Pa$ fails to be stable if and only if there exists \DM{a 1-PS $\lambda$} of $SL(W)$ such that $\lambda(t) \cdot A$ is bounded as $t\to \infty$ (as if this is the case  $\overline{(\lambda(\D{G}_m)\cdot A}\DM)\setminus  (\lambda(\D{G}_m)\cdot A)\ne \emptyset$). Take now a $(n+1)$-tuple of tensors $\omega_0, \DM{\ldots,} \omega_n$ spanning $A$. Since for $t\in \D{C}^*, \ \lambda(t) \cdot A$ is generated by $\lambda(t)\cdot \omega_i, i=0, \dots, n$, we see that $\Pa$ fails to be $\lambda$-semistable if and only if $0=\lim_{t\to \infty} \lambda(t) \cdot \omega_i, \ \forall\: i=0, \ldots, n$ that is, if and only if each generator is not $\lambda$-semistable. Similarly $\Pa$ is not $\lambda$-stable if and only if each $\lambda(t)\cdot\omega_i, \ \forall \: i=0, \ldots, n$ is bounded for $t\to \infty$ that is, if and only if each generator is not $\lambda$-stable.    
	The analysis of the stability for 1-PS reveals now to be much more easier since the action of these groups is always diagonalizable. This means that given 	
	a 1-PS $\lambda$ of $SL(W), \ \lambda$ 
	is uniquely determined by a decreasing six-tuple of complex numbers, $\lambda_1\ge \DM{\ldots,} \ge \lambda_6$, not all equal to zero, with $\sum_{i=1}^6 \lambda_i=0$ and such that $\lambda(t)=diag(t^{\lambda_1}, \dots, t^{\lambda_6})$ \DM{ (for} this reason we will often denote such a $\lambda$ with $(\lambda_1, \DM{\ldots,} \lambda_6)$, the 6-tuple of weights of its action).
	\DM{Adapting} the argument used by Wall in \DM{\cite{Wall}},we \DM{prove} the following:
	
	\begin{thm}\label{GIT}
	Let $\Pa\subset \Pw$ be a $n$-dimensional linear space and let $\omega_0,\ldots, \omega_n$ be a $(n+1)$-tuple of generators.
 \DM{We denote} by $M_{k}$ \DM{the skew-symmetric matrices} representing the \DM{forms} $\omega_k$\DM{, ${M_{k}}=(m_{ij}^k)$, \mbox{$\ 1\le i<j \leq 6, \ k=0, \dots, n$.} } 
		
		\begin{enumerate}
			\item  $\Pa$ is not \DM{semistable} if and only if, for some choice of coordinates on $W$, there \DM{exists an integer $1\le s \le 3$} such that  $m_{ij}^k=0\ \text{\DM{whenever}} \ 1\le i \le s, \ i<j\le 7-s$, \mbox{$0 \le k\le n$.}
			\item $\Pa$ is not stable if and only if, for some choice of coordinates on $W$, there \DM{exists an integer $1\le s \le 3$}
			such that, \mbox{$m_{ij}^k=0$} \DM{whenever} $1\le i \le s, \ i<j\le \DM6-s,$ \mbox{$0\le k\le n$}.
		\end{enumerate}
	\end{thm}

	\begin{proof}
	In the first place we recall that given $\Pa\subset \Pw$ as in the statement of the theorem, and 
	$\lambda=(\lambda_1,\DM{\ldots,} \lambda_6)$ a 1-PS of $SL(W)$, $\lambda(t)\cdot A$ is then generated by $\lambda(t) \cdot \omega_k$, tensors represented by the matrices ${(\lambda(t)M_k \lambda(t)^T)}_{ij}=t^{\lambda_i+\lambda_j} m_{ij}^k$, \mbox{$ k=0, \dots, n$.}
		
		\begin{enumerate}
			
			\item
			Suppose that $\Pa$ is \DM{unstable, that is, non-semistable}, and let $\lambda=(\lambda_1,\dots, \lambda_6)$ be a 1-PS for which we have $0 \in \overline{\lambda(\D{G}_m) \cdot A}$.  We thus have $0\in \overline{\lambda(\D{G}_m)\cdot \omega_k}$, \mbox{$ \forall\: k=0,\DM{\ldots,} n$} or, in other words, none among the generators is $\lambda$-semistable. There exists then an integer $s\in \{1,2,3\}$ such that $\lambda_i+ \lambda_{7-i}\ge 0$. Indeed, if this \DM{was not} the case, we would have $\sum_{i=1}^ 6 \lambda_i< 0$, \DM{which is} absurd. For such an $s$ we therefore have $\lambda_s+\lambda_{7-s}\ge 0$ and so, from the assumptions on the \DM{$\lambda_i$}s, $\lambda_i + \lambda_j\ge 0$, \mbox{ $\forall\: 1\le i\le s, \ \forall\: i<j\le 7-s$}. As we are assuming that $\Pa$ is not $\lambda$\DM{-semistable,} we conclude that $m_{ij}^k=0$ whenever $1\le i \le s, i<j\le 7-s$.
			\DM{Conversely,} if the coordinates $m_{ij}^k$ of $A$ \DM{satisfy} the hypotheses of the theorem\DM{,} we are able to construct explicitly a 1-PS of $SL(W)$ \DM{refuting} the semistability of $\Pa$. 
			This is the case for $\lambda \in \DM{\Hom_{\mathrm{Gr-Alg}}}(\D{G}_m, SL(W))$ acting with weights $\lambda_1, \dots, \lambda_6$ defined as follows: 
			$$\lambda_i=\begin{cases}6-s & 1 \le i \le s\\
			-1 & s+1\le i\le 7-s
			\\ s-7 & 8-s\le i\le 6\end{cases}$$
			
			\item
			Suppose now that $\Pa$ is \DM{nonstable}. This means that there exists a 1-PS $\lambda$ of $SL(W)$, \mbox{$\lambda=(\lambda_1,\DM{\ldots,} \lambda_6)$} for which $\forall\: k=0,\dots, n, \ \lambda(t)\cdot  \omega_k$ is bounded as $t\to \infty$. 
			Now, if ever for all \DM{integers} $s\in \{1,2,3\}$ we had $\lambda_s+\lambda_{6-s}\le 0$\DM, we would get $2\sum_{i=1}^5 \lambda_i \le 0$ hence $\sum_{i=1}^5 \lambda_i \le 0$. As $\sum_{i=1}^6 \lambda_i =0$ we should then have $\lambda_6 \ge 0$. Thus for every $1\le i\le 5$, \mbox{$\lambda_i\ge \lambda_6\ge 0$}, condition that can be satisfied if and only if all the \DM{$\lambda_i$'s} are equal to zero, a contradiction. \DM{Then there exists an} $s, \ 1\le s\le 3$, for which $\lambda_s+\lambda_{6-s}>0$, and consequently for all $i,j$ with $i\le s, \ i<j\le 6-s$, $\lambda_{i}+\lambda_{j}>0$. 
			As \DM{the non-stability} requires \DM{the} boundedness of $\lambda(t)\cdot \omega_k$ as $t\to \infty $ for every $k=0,\DM{\ldots,} n$, we must have $m_{ij}^k=0$ whenever $0\le k \le n, \ i\le s,\ i<j\le 6-s$.
			
			For the converse implication, consider $A\subset \bigwedge^2 W^*$ whose affine coordinates $m_{ij}^k$ in $V_{n+1}^*\otimes \bigwedge^2 W^*$ \DM{satisfy} the hypotheses of the proposition. Again, we are able to provide a \DM{1-PS} $\lambda$ for which $\Pa$ is not $\lambda$-stable. We can consider for example a \DM{1-PS} acting with the following weights:
			$$\lambda_i=\begin{cases}1 & 1 \le i \le s\\
			0 & s+1\le i\le 6-s \\ -1 & 6-s+1\le i\le 6\end{cases}$$
		\end{enumerate}
	\end{proof}
	
	\
	
	Applying theorem \ref{GIT} we see then that given a nonstable $\Pa\subset\Pw$, the $6\times 6$ skew-symmetric matrix of linear forms $M_{A}$ defined by $A$ 
	 is, for an appropriate choice of coordinates, of one of the following forms:
	\begin{equation*}{
		\begin{pmatrix}
		0 & 0 & 0 & 0 & 0 & * \\
		0 & 0 & * & * & * & * \\
		0 & * & 0 & * & * & * \\
		0 & * & * & 0& * & *\\
		0 & * & * & * & 0& *\\
		* & * & * & * & * & 0 
		\end{pmatrix}\hspace{2mm}(s=1) \hspace{2cm}
		\begin{pmatrix}
		0 & 0 & 0 & 0 & *& * \\
		0 & 0 & 0 & 0 & * & * \\
		0 & 0& 0 & * & * & * \\
		0 & 0 & * & 0& * & *\\
		* & * & * & * & 0& *\\
		* & * & * & * & * & 0 
		\end{pmatrix} \hspace{2mm} (s=2) }
	\end{equation*}
	\\
	\begin{equation*}{
		\begin{pmatrix}
		0 & 0 & 0 & *& * & * \\
		0 & 0 & 0 & * & * & * \\
		0 & 0 & 0 & * & * & * \\
		* & * & * & 0& * & *\\
		* & * & * & * & 0& *\\
		* & * & * & * & * & 0 
		\end{pmatrix}\hspace{2mm}(s=3)}
	\end{equation*}
	
	\vspace{1 cm}
	
	\DM{Similarly, for a} linear space $\Pa$ that is not \DM{semistable}, the linear space $A$ might be represented as a \DM{skew-symmetric matrix} of linear forms $M_{A}$ of one of these types:
	\begin{equation*}{
		\begin{pmatrix}
		0 & 0 & 0 & 0 & 0 & 0\\
		0 & 0 & * & * & * & * \\
		0 & * & 0 & * & * & * \\
		0 & * & * & 0& * & *\\
		0 & * & * & * & 0& *\\
		0 & * & * & * & * & 0 
		\end{pmatrix}\hspace{2mm}(s=1) \hspace{2cm}
		\begin{pmatrix}
		0 & 0 & 0 & 0 & 0& * \\
		0 & 0 & 0 & 0 & 0 & * \\
		0 & 0& 0 & * & * & * \\
		0 & 0 & * & 0& * & *\\
		0& 0 & * & * & 0& *\\
		* & * & * & * & * & 0 
		\end{pmatrix} \hspace{2mm} (s=2) }
	\end{equation*}
	\\
	\begin{equation*}{
		\begin{pmatrix}
		0 & 0 & 0 & 0& * & * \\
		0 & 0 & 0 & 0 & * & * \\
		0 & 0 & 0 & 0 & * & * \\
		0 & 0 & 0 & 0& * & *\\
		* & * & * & * & 0& *\\
		* & * & * & * & * & 0 
		\end{pmatrix}\hspace{2mm}(s=3)}
	\end{equation*}
	
	\
	
	\subsection{The stability criterion: a geometric formulation}\label{stab_geom}

In this section we show how it is possible to rephrase the stability in a "more geometric'' form (this characterization of stability will be the one we will mainly make use of).
 
To start with we observe that given any $s$-dimensional linear subspace $U\subset W$ and any form $\omega \in \bigwedge^2 W^*$, the restriction $\omega\restriction_U$ of $\omega$ to $U$:
\begin{align*}
\omega\restriction_U:U&\rightarrow W^*\\
u&\mapsto \omega(u, \cdot )
\end{align*}
has rank $s$ (so that $\dim \im(\omega\restriction_U)^{\perp}=6-s$) for $\omega$ general and rank strictly less than $s$ (so that $\dim \im(\omega\restriction_U)^{\perp}\ge 7-s$) whenever $\ker(\omega)\cap U \ne \{0\}$. 
The inclusion \mbox{$U\subset \im(\omega\restriction_U)^{\perp}$} holds if and only if $U$ is isotropic with respect to $\omega$, namely if and only if $\forall\ u, u' \in U$, we have $\omega(u,u')=0.$ 
Applying theorem \ref{GIT}, we get that a $n$-dimensional linear space $\Pa \subset \Pw$ is not stable whenever there exists a linear subspace $U$ of $W$ of dimension $s, \ 1\le s\le 3$, isotropic with respect to any tensor in $A$ and such that for any pair $\omega, \ \omega'$ of generic points in $A$ we have  ${\im(\omega\restriction_{U})}^{\perp}={\im (\omega'\restriction_{U})}^{\perp}$. If $\Pa$ is not even semistable, then 
\mbox{$\ker(\omega)\cap U\ne \{0\}, \ \forall \: \omega \in A$.}

From these observations we immediately deduce the following:
\begin{cor}\label{pfaff_ss} Let $\Pa \subset \Pw$ be an unstable $n$-dimensional linear space. Then $\Pa\subset \Pf$.
\end{cor}
\begin{proof}
The instability of $\Pa$ implies the existence of $U\subset W$, a linear space of dimension $s$, $1\le s\le 3$ such that 
$\ker(\omega)\cap U\ne \{0\}, \ \forall\: \omega \in A$. This condition clearly implies that $\rk(\omega)\le 4, \ \forall\:\omega\in A$ hence that $\Pa\subset \Pf$.
\end{proof}

\paragraph{Destabilizing subspaces of $\Pa^{\perp}$.}

\

Let's consider a non-stable $n$-dimensional linear space $\Pa$ and let $U$ be a $s$-dimensional subspace of $W$ preventing its stability. From what we have just explained, there exists then 
a linear space $U'\subset W$ of dimension at least $6-s$ containing $U$ and contained in any hyperplane $\omega(u, \cdot) \in W^*$ with $\omega \in A, \ u\in U$. 
For such a pair of linear spaces $U, \ U'$ we thus get that $\omega(U,U')=0, \ \forall \: \omega \in A$ and hence $\Pa\subset \D{P}(U\wedge U')^{\perp}$ (or equivalently $\D{P}(U\wedge U')\subset \Pa^{\perp}$).
Analyzing each possible value of $s$ we describe here in greater details, how unstable and non-stable linear systems might be characterized.

\begin{itemize}

\item \textbf{s=1.} 
For $s=1$, if $\Pa$ is nonstable there exists a vector $u\in W$ such that all hyperplanes in $W$ defined by the equations $\omega(u,\cdot)=0, \ \omega\in A$, intersect along a linear space $U'$ of dimension   $\ge 5$. $\Pa$ fails to be semistable when $\dim U'=6$, namely when $u \in \ker(\omega)$, $\forall \omega \in A$ so that $U'=W$.
 $\D{P}(u\wedge U')$ is a linear subspace of $\Gr(2,U')\subset \Gr(2,W)$ having dimension bigger then or equal to 3 (more precisely of dimension 3 if $\Pa$ is strictly semistable and of dimension 4 if $\Pa$ is unstable); its orthogonal is a linear subspace of $\Pw$ containing $\Pa$ that can be generated as follows:
$$\D{P}( u\wedge U')^{\perp}= \langle \D{P}(\bigwedge^2 u^{\perp}), \ \D{P}(\bigwedge^2 {U''}^{\perp})\rangle$$
where $U''$ is any hyperplane in $U'$ not containing $u$.
We observe that if $\Pa$ is unstable,  $U'=W$, $\bigwedge ^2 {U''}^{\perp}= 0$, hence $\D{P}(u\wedge U')^{\perp}=\D{P}(\bigwedge^2 u^{\perp})$ is a 9 dimensional linear space isomorphic to $\D{P}(\bigwedge^2 \D{C}^5)$. If $\Pa$ is strictly semistable ${\D{P}(u\wedge U')}^{\perp}$ is the ten dimensional linear space spanned by $\D{P}(\bigwedge ^2 u^{\perp})$ and by the point \mbox{$\D{P}(\bigwedge^2 U''^{\perp})\in \Gr(2,W^*)$}.

\item \textbf{s=2.} For $s=2$ there exist a 2-dimensional linear space $U$ and a linear space $U'$ 
of dimension $\ge 4$ ($U'$ will have dimension strictly bigger then 4 whenever $\Pa$ is unstable) such that $\omega(U,U')=0,\ \forall\: \omega \in A$. 
In this case we have that any \mbox{$U''\simeq\D{C}^3$} contained in $U'$ and containing $U$ is isotropic with respect to any form in $A$ and therefore satisfies $\D{P}^2\simeq\Gr(2,U'')\simeq\D{P}(\bigwedge^2 U'')\subset {\Pa}^{\perp}$. 
Planes $\D{P}(\bigwedge^2 U'')$ of this kind span $\D{P}(U\wedge U')$, a linear space of dimension $\ge 4$ (its dimension is $\ge 6$ whenever $\Pa$ is unstable) satisfying $\Pa\subset \D{P}(U\wedge U'')^{\perp}$.

\item \textbf{s=3.} For $s=3$, $\D{P}(A)$ is not stable if there exists a linear space 
$U$ of dimension at least 3 that is completely isotropic with respect to any form in $A$. $\Pa$ is strictly semistable if $\dim(U)=3$; for such an $U$, $\D{P}(\bigwedge^2 U)\simeq \Gr(2,U)\simeq \D{P}^2$ is a plane satisfying $\Pa \subset \D{P}(\bigwedge^2 U)^{\perp}\simeq \D{P}^{11}$.
$\Pa$ is unstable whenever 
$\dim(U)=4$; 
 in this case $\Pa$ is contained in the 8-plane ${\D{P}(\bigwedge^2 U)}^{\perp}$, a linear space 
 coinciding with $\D{T}_{\omega_U}\Gr(2,W^*)$, where $\omega_{U}\in \Gr(2,W^*)$ is the point defined by ${U}^{\perp}$. 
\end{itemize}

\begin{rmk}\label{2_then_3} As remarked while treating the case $s=2$, we see that if $\Pa$ is ``destabilized'' by a couple of spaces $U<U'$ of dimension 2 and 4, then it is ``destabilized" by any hyperplane $U''$ of $U'$ containing $U$.
		Indeed given any tensor $\omega\in {U\wedge U'}^{\perp}$ and any vector $u \in U', \ u\not\in U$, since  $\omega(u,u)=0$, the 3-dimensional space $U'':=\DM{\langle}U,u\DM{\rangle}$ always satisfies $\omega(U'',U'')=0$ so that $\D{P}(\bigwedge ^2 U'')\subset \Pa^{\perp}$. 
		Otherwise we can argue saying that the existence of the aforementioned couple $U, U'$ implies that $\Pa^{\perp}\cap \DM{\Gr}(2,U')$ defines an hyperplane section $H\cap \DM{\Gr}(2,U')$ of $\DM{\Gr}(2,U')\simeq \Gr(2,4)$ by an hyperplane $H\in \DM{\Gr}(2,U^*)$.
		Such an hyperplane section is a 4\DM{-dimensional} quadric of rank 4 that always contains a plane isomorphic to $\DM{\Gr}(2,3)$.
	\end{rmk}

	\

Another useful corollary that we deduce straightaway from theorem \ref{GIT} is the following:
\begin{cor}\label{unst_sub} If $\Pa\subset \Pw$ is not (semi)stable then every linear subspace of $\Pa$ is not (semi)stable. 
\end{cor}

\section{Destabilizing subspaces of dimension  $\le 2$}
Given a stable linear space $\Pa \subset \Pf$, we study now how the stability assumption imposes restrictions on the dimension of $\Pa$ and on the behavior of linear subspaces of $\Pa$. 
We prove the following theorem:
\begin{thm}\label{dim_le3}Let $\Pa\subset \Pf$ be a $n$-dimensional stable linear space. Then the following hold:
	\begin{itemize}
		\item $n\ge 3$;
		\item a general plane $\Pb \subset \Pa$ is such that $\Pb \cap \Gr(2,W^*)=\emptyset$;
		\item any plane $\Pb \subset \Pa$ such that $\Pb \cap \Gr(2,W^*)=\emptyset$ is $SL(W)$-equivalent to the plane $\pi_g$:
		$$\pi_g=\langle e_1\wedge e_4 + e_2\wedge e_5, \ e_1 \wedge e_6+ e_3\wedge e_5, \ e_2\wedge e_6 - e_3\wedge e_4 \rangle$$
		for a basis $e_1,\ldots e_6$ of $W^*$;
		\item $n\le 5$.
	\end{itemize} 
	\end{thm}
We will prove the theorem in several steps.
In the first place we look at the behavior of the intersection of a linear space satisfying the hypotheses of the theorem with the Grassmannian $\Gr(2,W^*)$. More generally, whenever dealing with subspaces $\Pa$ of the Pfaffian hypersurface, it is indeed natural to ask whether these intersect the singular locus of $\Pf$, that is $\Gr(2,W^*)$, and consequently to determine how this intersection might affect their stability.
The answer to the first question is positive whenever $n:=\dim(\Pa)\ge 3$. 
It is indeed proven in \cite{MM} the following:
\begin{prop}[Manivel-Mezzetti (\cite{MM}, Cor.11)] There exists no $\D{P}^3$ of skew-symmetric matrices of order six and constant rank four.
\end{prop}
The main result of loc.cit is the complete classification (up to the $PGL(W)$ action) of planes of skew-symmetric forms all having rank equal to four.
Starting from a plane of such a kind, it is then easy to construct explicitly 3-dimensional linear subspaces of $\Pf$ intersecting $\Gr(2,W^*)$ in finitely many points. Therefore for any $\Pa\simeq \D{P}^3, \ \Pa\subset \Pf$, we have $\dim(\Pa\cap \Gr(2,W^*))\ge 0$ so that, more generally we get: 

\begin{cor}\label{cor_MM}
Let $\Pa \subset \Pf$ be a linear space of dimension $n\ge 3$. Then \mbox{$\Pa\cap \Gr(2,W^*)$} is non empty and has codimension (in $\Pa$) at most $3$.
\end{cor}
We show now that if furthermore $\Pa$ is stable, $\codim(\Pa\cap \Gr(2,W^*))$ must be exactly 3.
\begin{prop}\label{codim_int}
Let $\Pa \subset \Pf$ be a stable linear space of dimension $n$ such that $\Pa \cap \Gr(2,W^*)\ne \emptyset$. Then $\Pa \cap \Gr(2,W^*)$ has codimension (in $\Pa$) 
at least 3 and each of its components of maximal dimension is smooth.
\end{prop}
\begin{proof}
    Let $\Pa$ be a linear space satisfying the hypotheses of the proposition and denote by $X$ the intersection $X:=\Pa \cap \Gr(2,W^*)$. 
    We notice that whenever we take a point $\omega \in X$ and a linear subspace $V_{\omega}$ of $W^*$ complementary to $L_{\omega}$, we get a decomposition of $\D{P}(\bigwedge^2 W^*)$ as :
    $$\D{P}(\bigwedge^2 W^*)=\langle \D{T}_{\omega}\Gr(2,W^*), \D{P}(\bigwedge^2 V_{\omega})\rangle.$$
    The projective tangent space to $\Gr(2,W^*)$ at $\omega$ is indeed the 8 dimensional linear space spanned by points corresponding to lines in $\D{P}(W^*)$ meeting $\D{P}(L_{\omega})$. Points in $\Gr(2,V_{\omega})$, that parametrize lines in $\D{P}(V_{\omega })$, span instead the 5-dimensional linear space $\D{P}(\bigwedge^2 V_{\omega})$; since this latter is disjoint from $\D{T}_{\omega}\Gr(2,W^*)$ we obtain the aforementioned decomposition.   
    We consider now the linear projection $\pi_{\omega}$ 
    from $\D{T}_{\omega}\DM{\Gr}(2,W^*)$ to  $\D{P}(\bigwedge^2 V_{\omega})$:
    	$$\pi_{\omega}: \D{P}(\bigwedge^2 W^*) \dashrightarrow \D{P}(\bigwedge^2 V_{\omega}).$$
    If ever there exists a point $\omega\in X$ such that $\dim(\D{T}_{\omega}X)\ge n-2$, we would get that $\pi_{\omega}(\Pa)$, the image of $\Pa$ under $\pi_{\omega}$, is a linear space of dimension $m\le 1$. (Here we use the convention $\dim(\pi_{\omega}(\Pa))=-1$ if ever \mbox{$\pi_{\omega}(\Pa)=\emptyset$}).
    Since moreover we are supposing that $\Pa \subset \Pf$, we must have $\Pa \subset Q_{\omega}$, where $Q_{\omega}\in |\mathcal{Q}|$ denotes the quadric tangent cone to $\Pf$ at $\omega$ (as described in section \ref{sect_hw}). Since $\Sing(Q_{\omega})=\D{T}_{\omega}\Gr(2,W^*)$, this condition is fulfilled if and only if $\pi_{\omega}(\Pa)\subset \Gr(2, V_{\omega})$. 
    This last inclusion 
    implies that \mbox{${\Pa}^{\perp}\cap \Gr(2, {L_{\omega}}^{\perp})$} is a quadric hypersurface of rank $4-m$ of the linear space \mbox{${\Pa}^{\perp}\cap  \D{P}(\bigwedge^2 {L_{\omega}}^{\perp})={\pi_{\omega}(\Pa)}^{\perp} \cap  \D{P}(\bigwedge^2 {L_{\omega}}^{\perp})\simeq \D{P}^{4-m}$}. 
 As such a linear section of the Grassmannian $\Gr(2,{L_{\omega}}^{\perp})\simeq \Gr(2,4)$ always contains a plane isomorphic to $\Gr(2,3)$, there exists a linear space $U< {V_{\omega}}^{\perp}$ of dimension 3 such that \mbox{$\D{P}(\bigwedge^2 U)\subset \Pa^{\perp}$}.
    We can therefore conclude that $\Pa$ can not be stable. The stability of $\Pa$ implies thus that $\forall\: \omega \in X, \ \dim(\D{T}_{\omega}X)\le n-3$; as a consequence $X$ must have codimension at least 
    3 and moreover each $n-3$ dimensional component of $X$ must be smooth. 
    \end{proof}

From this proposition we deduce what follows: given a $n$-dimensional stable linear space $\Pa\subset \Pf$ if ever $n \le 2$, $\Pa$ must then have constant rank 4, if $n \ge 3$ instead, $\Pa\cap \Gr(2,W^*)$ has codimension 3 (in $\Pa$) so that $\Pa$ must contain a plane of constant rank 4. 
Both cases can be approached  studying at first planes of tensors of constant rank 4.

\subsection{ Destabilizing planes of constant rank 4}
2-dimensional linear spaces $\Pb$ parameterizing tensors all having rank equal to 4 have been classified, up to the action of $PGL(W)$, in \cite{MM}.  
The authors proved that such a $\Pb$ belongs to the $PGL(W)$-orbit of one of the following planes 
\begin{equation}\label{list_plane}\begin{split}
\pi_g=&\DM{\langle} e_1\wedge e_4 + e_2\wedge e_5, e_1\wedge e_6+ e_3\wedge e_5, e_2\wedge e_6 - e_3 \wedge e_4\DM{\rangle}\\
\pi_t=&\DM{\langle} e_1\wedge e_3 +e_2\wedge e_4, e_1\wedge e_4+e_2\wedge e_5, e_1\wedge e_5+ e_2\wedge e_6\DM{\rangle}\\
\pi_p=&\DM{\langle}e_1\wedge e_4+e_2\wedge e_3, e_1\wedge e_5+ e_3\wedge e_4, e_1\wedge e_6 +e_2\wedge e_4\DM{\rangle}\\
\pi_5=&\DM{\langle}e_1\wedge e_4+e_2\wedge e_3, e_1\wedge e_5+e_2\wedge e_4, e_2\wedge e_5 + e_3\wedge e_4\DM{\rangle}
\end{split}\end{equation}

	\begin{prop}\label{plane_unst}Let $\Pa\subset \Pf$ be a linear space of dimension $n\ge 2$ and let \mbox{$\Pb\simeq \D{P}^2$}, $\Pb\subset \Pa$ be a plane of tensors of constant rank 4. 
		If $\Pb$ is $PGL(W)$\DM{-equivalent} either to $\pi_t, \ \pi_p, \ \pi_5$, then $\Pa$ can't be stable.
	\end{prop}
For the proof of the proposition we will need the following lemma:
\begin{lem}\label{lemma_gamma} Let $\Pa\subset \Pf$ be a $n$-dimensional linear space and let $\Pb\subset \Pa$ be a linear space of tensors all having rank equal to four. 
Then $\Pa\subset {\langle\gamma(\Pb)\rangle}^{\perp}$.
\end{lem}
\begin{proof} Let $\Pa, \ \Pb$ be a pair of linear spaces satisfying the hypotheses of the lemma. Since all tensors in $\Pb$ have rank 4 (i.e. $\Pb\cap \Gr(2,W^*)=\emptyset$), $\Pb$ lies entirely in the regular locus of $\gamma$, the Gauss map defined by $\Pf$. The inclusion $\Pa\subset \Pf$ implies that $\forall\: \omega \in \Pa, \ \D{P}(A)\subset \D{T}_{\omega}\Pf$, so that, in particular, $\forall\: \omega \in \Pb$ we must have that $\D{P}(A)\subset \D{T}_{\omega}\Pf={\gamma(\omega)}^{\perp}$. This means that $\D{P}(A)$ must be contained in all hyperplanes belonging to $\gamma(\D{P}(B))$ and since $\D{P}(A)$ is a linear space this happens if and only if $\D{P}(A)\subset {\langle \gamma(\Pb)\rangle}^{\perp}$.
\end{proof}

	\begin{proof}[Proof of prop. \ref{plane_unst}]
	Applying lemma \ref{lemma_gamma}, we get an inclusion $\Pa\subset {\langle \gamma(\Pb)\rangle}^{\perp}$, hence $\langle \gamma(\Pb)\rangle \subset {\Pa}^{\perp}$. 
		By means of the classification (\ref{list_plane}), we can compute directly $\gamma(\Pb)$ (a detailed description of these images can be found in \cite{MM}) and see that whenever $\Pb$ is $PGL(W)$\DM{-equivalent} to $\pi_p, \ \pi_t$ or $\pi_5$ we can always find a linear subspace of \mbox{$\DM{\langle}\gamma(\Pb)\DM{\rangle}\subset {\Pa}^{\perp}$} that prevents the stability of $\Pa$.  
		
		\begin{itemize}
			\item $\Pb\in PGL(W)\cdot\pi_t$. A linear space of this kind is always contained in a tangent space $\D{T}_{\omega}\DM{\Gr}(2,W^*)$ to $\DM{\Gr}(2,W^*)$ at a point $\omega\in \DM{\Gr}(2,W^*)$ (for the choice of coordinates in (\ref{list_plane}), $\pi_t \subset \D{T}_{(e_1\wedge e_2)}\DM{\Gr}(2,W^*)$.)
			In this case $\gamma|_ {\Pb}$ is defined by the complete linear system $|\mathcal{O}_{\Pb}(2)|$ and $\gamma(\Pb)$ is a Veronese surface contained in $\DM{\Gr}(2, L_{\omega}^{\perp})\subset \D{P}(\bigwedge^2 {L_{\omega}}^{\perp},  )\simeq \D{P}^5$. Thus $\D{P}(\bigwedge^2 {L_{\omega}}^{\perp})\subset \Pa^{\perp}$, i.e. the restriction of every tensor in $A$ to $L_{\omega}^{\perp}$ is zero. From the discussion held in section \ref{stab_geom} we conclude that $\D{P}(A)$ can't even be semistable.

			\item $\Pb \in PGL(W) \cdot \pi_p$. In this case the plane $\Pb$ contains a pencil of \textit{special lines}, namely lines of tensors of constant rank 4 entirely contained in $\D{P}(\bigwedge W')$, for a 5-dimensional subspace $W'$ of $W^*$. 
			For any pair $l_1, \ l_2$ of generators of this pencil, there exists a pair of points $\omega_i \in \DM{\Gr}(2, W^*), \omega_i \in \D{T}_{\omega_j}\Gr(2,W^*),\ i\ne j$ such that $l_i\subset \D{T}_{\omega_i}\DM{\Gr}(2,W^*), \ i=1, 2 $. Denote by $U_{12}<W$ the 3-dimensional space $\langle L_{\omega_1}, L_{\omega_2}\rangle^{\perp}$ (For $\D{P}(B)=\pi_p$ as in (\ref{list_plane}) we can take for example the lines 
			\mbox{$l_1= \overline{e_1\wedge e_4+e_2\wedge e_3, e_1\wedge e_5+ e_3\wedge e_4}$}, $l_2=\overline{e_1\wedge e_4+e_2\wedge e_3,  e_1\wedge e_6 +e_2\wedge e_4}$, so that $\omega_1=e_1\wedge e_3, \ \omega_2= e_1\wedge e_2$ and $U_{12}=\DM{\langle}e_1,e_2,e_3\DM{\rangle}^{\perp}$).
			The map $\gamma|_ {\Pb}$ is still a Veronese embedding; the span of the Veronese surface $\gamma(\Pb)$ contains the plane $\D{P}(\bigwedge^2 U_{12})\simeq \DM{\Gr}(2,3)$ that prevents the stability of $\Pa$.
			\item $\Pb \in PGL(W)\cdot \pi_5$. In this case there exists a vector $u\in W$ such that we have an \DM{inclusion} \mbox{$\D{P}(B)\subset \D{P}(\bigwedge^2 u^{\perp})\simeq \D{P}(\bigwedge^2 \D{C}^5)$}; this time $\DM{\langle}\gamma(\Pb)\DM{\rangle}$ is equal to $\D{P}(v\wedge W)\simeq \D{P}^4$, the Schubert variety of 2-dimensional subspaces of $W$ containing $u$
				(for the choice of coordinates in (\ref{list_plane}), $u$ will be defined by the intersection of the hyperplanes $e_1,\DM{\ldots,} e_5$).
			In this \DM{situation} the entire space $\Pa$ must thus be contained in $\D{P}(\bigwedge^2 v^{\perp})$; therefore it can't  even be semistable.  
		\end{itemize}
	\end{proof}

\begin{rmk}\label{plane_not_stab}
From the proof of the previous proposition we can observe that the planes of type $\pi_t,\ \pi_5, \pi_p$ are not stable and that moreover $\pi_t$ and $\pi_5$ are even unstable.
	\end{rmk}
From \ref{plane_unst} we learn that any stable linear subspace of $\Pf$ of dimension greater than or equal to two can only contain planes of constant rank 4 that lie in the $PGL(W)$ orbit of $\pi_g$. Adopting the nomenclature of \cite{MM} we will refer to these planes as \textit{planes of general type}. 
Next section is devoted to a more detailed analysis 
of these planes. 
\subsection{Planes of general type}

Throughout the rest of the section $\Pb$ will denote a plane of constant rank 4 of general type.
These planes have a \DM{convenient} characterization that we will \DM{use} in the proofs of our main results. 
Indeed, following \cite{MM}, any plane $\Pb$ of general type can be characterized as the locus of tensors of the form:
$$ \omega= x\wedge f(y) - y\wedge f(x), \hspace{2mm} \text{for}\  x,y \in C,$$
with $C$ a 3 dimensional subspace of $W^*$ and $f$ a linear isomorphism from $C$ to a linear space $D\subset W^*$ disjoint from $C$. 

The restriction of the Gauss map to $\Pb$:

$$\gamma\restriction_{\Pb}:\Pb\longrightarrow \D{P}(\bigwedge^2 W)$$
is defined by the complete linear system $|\mathcal{O}_{\Pb}(2)|$. The Veronese surface $\gamma(\Pb)$ is the the variety of lines of the form $\overline{u f^T(u)}$ where:
$$f^T: D^*\to C^*, \hspace{2mm} D^*\simeq C^{\perp}, \ C^*\simeq D^{\perp}$$
is the transpose of $f$.

\vspace{2mm}

Throughout the rest of the section we fix the basis $e_1, \DM{\ldots,} e_6$ of $W^*$ in such a way that $\Pb$ is written as \mbox{$\Pb=\pi_g=\DM{\langle}\omega_0, \omega_1, \omega_2\DM{\rangle}$} with the generators $\omega_i, \ i=0,1,2$ of the form appearing in (\ref{list_plane}).

We will now show that even the plane $\Pb$ is not stable, describing explicitly which linear subspaces of $W$ prevent its stability. 
To start with, we prove the following:
\begin{prop}\label{stab_s1} There is no pair $u, \ U' $ with $u\in W, \ U'<W, \ \DM{\dim}(U')=5$ such that $\D{P}(u\wedge U' )\subset \Pb^{\perp}$.
\end{prop}
\begin{proof} 
	Let $u_1, \DM{\ldots,} u_6$ be the dual basis to $e_1, \DM{\ldots,} e_6$. 
	As we have already remarked in section \ref{stab_geom}, the existence of a pair $u, \ U'$ as in the statement of the proposition is equivalent to the existence of a vector $u\in W$ such that all linear forms $\omega(u,\cdot )$, with $\omega \in B$, vanish along a 5-dimensional linear space $U'<W$. 
	We prove that this can never \DM{occur}. Indeed, take an arbitrary non-zero vector \mbox{$u\in W, \ u=\sum_{i=1}^6 \alpha_i u_i$}. In the basis $e_1, \DM{\ldots,} e_6$, the linear forms $\omega_0(u, \cdot), \ \omega_1(u,\cdot), \ \omega_2(u,\cdot)$ can be written as:
	$$ \omega_0(u, \cdot)= \alpha_1 e_4+ \alpha_2 e_5 - \alpha_4 e_1-\alpha_5 e_2, \hspace{3mm} \omega_1(u, \cdot)=\alpha_1 e_6+ \alpha_3 e_5 -\alpha_6 e_1-\alpha_5 e_3$$
	$$\omega_2(u, \cdot)= \alpha_2 e_6+ \alpha_4 e_3 - \alpha_6 e_2- \alpha_3 e_4.$$
	
	The intersection of the hyperplanes having equations $\omega_i(u, \cdot)=0, \ i=0,1,2,$ 
	has dimension greater than or equal to five if and only if the linear subspace of $W^*$ generated by $\omega_0(u, \cdot), \omega_1(u,\cdot), \omega_2(u,\cdot)$ has dimension at most one or equivalently, if and \DM{only if} the matrix
	
	$$\begin{pmatrix} -\alpha_4&-\alpha_5&0&\alpha_1&\alpha_2&0\\
	-\alpha_6&0&-\alpha_5&0&\alpha_3&\alpha_1\\
	0&-\alpha_6&\alpha_4&-\alpha_3&0&\alpha_2
	\end{pmatrix}$$
	has rank 1. 
	But we can compute directly that this never happens.

\end{proof}

Because of remark (\ref{2_then_3}), a \DM{n-dimensional} linear space $\Pa\subset \DM{\Pf}$ containing the plane $\Pb$ is thus not stable if and only if there exists a 3-dimensional space $U$ isotropic with respect to any tensor in $A$ (and consequently with respect to any tensor in $B$.) 
Every 3-dimensional \DM{subspace} of $W$ isotropic with respect to any form in $B$ can be characterized as follows:

\begin{lem}\label{plane_isotropic}
	Every 3-dimensional subspace $U$ of $W$ such that $\Pb \subset \D{P}(\bigwedge^2 U)^{\perp}$ might be written as:
	
	$$U=\DM{\langle}\alpha_3 u_3+ \alpha_6 u_6, \alpha_2 u_2+\alpha_4 u_4, \alpha_1u_1+\alpha_5 u_5\DM{\rangle}$$
	
	with $([\alpha_3:\alpha_6], [\alpha_2:\alpha_4], [\alpha_1:\alpha_5])\in \D{P}^1 \times \D{P}^1 \times \D{P}^1$ , satisfying:
	\begin{align}\label{quadr-eq}
	\alpha_2\alpha_5-\alpha_4\alpha_1=&0,\\
	\alpha_3\alpha_5-\alpha_6\alpha_1=&0,\\
	\alpha_6\alpha_2+\alpha_3\alpha_4=&0.
	\end{align}
\end{lem}

\begin{proof}Denote by $\omega_0, \omega_1, \omega_2$ the 3 generators of $\Pb$ appearing in (\ref{list_plane}) and let $U<W$ be a 3-dimensional linear space isotropic to each $\omega_i, \ i=0,1,2$.
	The first thing that we can deduce is that \DM{every such} linear space $U$ is necessarily spanned by 3 independent vectors in $W$, $k_1,k_2,k_3$ with $k_i\in \DM{\ker}(\omega_i)$.
	This is due to the fact that given a rank 4 tensor $\omega$, if $U$ is a 3-dimensional linear space isotropic with respect to it, then \mbox{$U\cap \DM{\ker}(\omega)\ne 0$.} Indeed, if by contradiction $\ker(\omega)\cap U=\emptyset$, we would have that the restriction of $\omega$ to $\langle U,\ker(\omega)\rangle$, a 5-dimensional vector space, is zero; but this is clearly not possible since $\rk(\omega)=4$.
	Therefore $U\cap \DM{\ker}(\omega_i)\ne 0 \ \forall\: \omega_i, \ i=1,2,3$.
	It's easy to compute directly that $\langle \ker(\omega_0),\ker(\omega_1),\ker(\omega_2)\rangle=W$, hence
	 we conclude that we can find a 3-uple of vectors in $W$, $k_0,\ k_1,\ k_2$ with $k_i, \in \DM{\ker}(\omega_i)$ spanning the space $U$. 
	For our choice of coordinates, we have $\DM{\ker}(\omega_0)=\DM{\langle}u_3,u_6\DM{\rangle}, \ \DM{\ker}(\omega_1)=\DM{\langle}u_2,u_4\DM{\rangle}$, \mbox{$ \DM{\ker}(\omega_2)=\DM{\langle}u_1,u_5\DM{\rangle}$}, hence 3-dimensional spaces isotropic with respect to any tensor in $\Pb$ belong to the family of linear spaces of the form: 
	$$U=\DM{\langle}\alpha_3 u_3+ \alpha_6 u_6, \alpha_2 u_2+\alpha_4 u_4, \alpha_1u_1+\alpha_5 u_5\DM{\rangle}$$
	
	with $([\alpha_3:\alpha_6], [\alpha_2:\alpha_4], [\alpha_1:\alpha_5])\in \D{P}^1 \times \D{P}^1 \times \D{P}^1.$

	Imposing the conditions $\omega_i\restriction_U\equiv 0$, we get the 3 equations (2),(3), (4).
	To see this we embed $\D{P}(\ker(\omega_1))\times \D{P}(\ker(\omega_2))\times \D{P}(\ker(\omega_1))\simeq \D{P}^1\times \D{P}^1 \times \D{P}^1$ in $\D{P}^{11}\subset \Pw$ by means of the morphism:
	\begin{align*}
	\D{P}(\ker(\omega_0)) \times \D{P}(\ker(\omega_1)) \times \D{P}(\DM\ker (\omega_2))&\overset{\Upsilon}\longrightarrow \D{P}(\bigwedge^2 W)\\
	(k_1,k_2,k_3)& \mapsto  k_1\wedge k_2+k_1\wedge k_3+ k_2\wedge k_3
	\end{align*}

	 ($\Upsilon$ is the morphism defined by the linear system of divisors of type $(1,1,0), \ (1,0,1), \ (0,1,1)$).
	 A linear space $U=\langle k_1, k_2, k_3\rangle$ generated by vectors $k_1=\alpha_3 u_3+ \alpha_6 u_6,\ k_2=\alpha_2 u_2+\alpha_4 u_4 $, \mbox{$ k_3=\alpha_1u_1+\alpha_5 u_5$}, 
	 satisfies $\omega_i\restriction_{U}\equiv 0, \ i=0,1,2,$ if and only if $\Upsilon(k_1, k_2, k_3)$ annihilates each $\omega_i$.
	 In other words 
	the conditions $\omega_i \in \D{P}(\bigwedge^2 U)^{\perp}, i=0,1,2$ define 3 hyperplane sections of $\D{P}^{11}$  that restricted to $\im(\Upsilon)$ give the equations (quadratic in the $\alpha_i$s ) appearing in the statement.
\end{proof}

From this result we deduce:

\begin{cor}\label{unst_le2} There exist no stable linear systems of generic rank 4 and of dimension less than or equal to 2.
\end{cor}
\begin{proof} As we have already remarked, from proposition \ref{codim_int} we know that a stable linear space $\Pa \subset \Pf$ having dimension $\le 2$ should necessarily have constant rank equal to 4. But from proposition \ref{plane_unst}
	and lemma \ref{plane_isotropic} we get that no plane of constant rank 4 is stable, therefore by corollary \ref{unst_sub} the same hold for any linear space of constant rank 4 and of dimension $<2$.
	\end{proof}

From what we have proved until now we see that a stable $\Pa\subset \Pf, \ \Pa \simeq \D{P}^n$ can be generated as
$\Pa=\langle \Pb, \omega_3, \ldots, \omega_{n}\rangle$ with $\Pb$ a plane of general type and $\omega_i$ belonging to $\Gr(2,W^*)$ for $i=3, \ldots n$.
We will now show that in order to effectively get an inclusion $\Pa \subset \Pf$, the tensors $\omega_i, \ i=3,\ldots n,$ must belong to a 3-dimensional rational scroll $Z$. 

\paragraph{Construction of the scroll \textbf{$Z\simeq S_{(2,2,2)}$}}

\

Let $\Pa=\langle \Pb, \omega_3, \dots, \omega_{n}\rangle$ be a linear space as described above. 
We denote by $\Lambda$ the 8-dimensional linear space defined by $\Lambda:=\langle\gamma(\Pb)\rangle^{\perp}$. By lemma \ref{lemma_gamma} we get an inclusion  $\Pa \subset \Lambda$ that imposes that each $\omega_i, \ i=3,\dots n$ is a rank 2 tensor belonging to $\Lambda$. We study then the intersection \mbox{$\DM{\Gr}(2,W^*)\cap \Lambda $.} 
Since $\Pb$ is a plane of general type, it is uniquely determined by a triple $f,C,D$, with $C,\ D$ \DM{disjoint subspaces} of $W^*$ of dimension 3 and $f$ an isomorphism $f:C\xrightarrow{\sim}D$. Every $\omega \in \Pb$ has the form:
$$ \omega= x\wedge f(y) -y\wedge f(x), \hspace{3mm} \exists \ x,y \in C$$
and this defines an isomorphism $\rho$:
\begin{align*}
\rho : \Pb\overset{\sim}{ \longrightarrow}& \DM{\Gr}(2,C)\\
x\wedge f(y)-y\wedge f(x)\mapsto & x\wedge y
\end{align*}
Consider now the morphism:
\begin{align*}
\psi': \DM{\Gr}(2,C) \times \D{P}^1 \longrightarrow & \DM{\Gr}(2,W^*)\\
(x\wedge y, [t_0:t_1]) \mapsto & (t_0(x)+t_1(f(x)))\wedge (t_0(y)+t_1(f(y))),
\end{align*}
and consequently the map:
\begin{align*}
\psi: \Pb\times \D{P}^1 \longrightarrow & \DM{\Gr}(2,W^*)\\
(\omega, [t_0:t_1]) \mapsto & \psi'((\rho(\omega), [t_0:t_1])); 
\end{align*}
denote by $Z$ the variety $Im(\psi)=Im(\psi')$.

We see from the definition of the maps $\psi$ and $\psi'$ that $Z$ is a rational normal scroll $S_{(2,2,2)}$, a degree 6 variety of dimension 3, non-degenerate in $ \Lambda$ (hence a minimal variety of dimension 3 in $\Lambda$).
$Z$ has the structure of a conic fibration over $\Pb$, 
$$Z\xrightarrow{\pi} \D{P}(B), \hspace{3mm} \mathfrak{C}_{\omega}:=\pi^{-1}(\omega)=\psi(\{\omega\}\times \D{P}^1).$$
Note that if $\rho(\omega)=x\wedge y$, the conic $\mathfrak{C}_{\omega}$ is the locus: $$\mathfrak{C}_{\omega}=\pi^{-1}(\omega)=t_0^2 (x\wedge y) + t_0t_1 (x\wedge f(y) +f(x)\wedge y) + t_1^2 (f(x)\wedge f(y)).$$
$Z$ can also be described as a family of planes over $\D{P}^1$.
Consider indeed the three conics \mbox{$\mathfrak{C}_{\omega_i}\subset  \Lambda$}, $i=0,1,2$; these are 3 conics lying in 3 disjoint planes and moreover given any point $\omega \in \mathfrak{C}_{\omega_i}, \ \D{T}_{\omega}\DM{\Gr}(2,W^*)\cap \mathfrak{C}_{\omega_j}, \ i\ne j$ consists of just one point.
This means that once we have fixed an isomorphism $\phi_0:\D{P}^1\to \mathfrak{C}_{\omega_0}$, we \DM{uniquely determine} isomorphisms $\phi_i:\D{P}^1\to \mathfrak{C}_{\omega_i}, \ i=1,2$ sending a point $p\in \D{P}^1$ to $\D{T}_{\phi_0(p)}\DM{\Gr}(2,W^*)\cap \mathfrak{C}_{\omega_i}$.
$Z\simeq S_{(2,2,2)}$ is thus the rational scroll obtained as:
$$ Z= \bigcup_{p\in \D{P}^1} \langle\phi_0(p), \phi_1(p), \phi_2(p)\rangle, \hspace{5mm}\langle\phi_0(p), \phi_1(p), \phi_2(p)\rangle=\psi(\D{P}(B)\times \{p\}).$$

\

As every point in the Veronese surface $\gamma(\Pb)$ is of the form $u\wedge f^T(u)$, where $f^T$, \mbox{$f^T:D^*\to C^*$} is the transpose of $f$, we see that by construction we have an inclusion \mbox{$Z\subset \DM{\Gr}(2, W^*)\cap \Lambda$.}
In order to verify that $Z$ is effectively equal to $\DM{\Gr}(2,W^*)\cap \Lambda$, we first observe that, as $Z\simeq S_{(2,2,2)}$, $Z$ is the base locus of a linear system of quadrics on $\Lambda$ having dimension 
$h^0(\mathcal{I}_{Z/\Lambda}(2))-1=14$. 

Consider now $\mathcal{Q}\simeq \bigwedge^4 W$ the linear system of Pl\"ucker's quadrics on $\Pw$. The variety $\Lambda\cap \DM{\Gr}(2,W^*)$ is the base locus of $|\mathcal{Q}|_{\Lambda}|$ and the inclusion $Z\subset \DM{\Gr}(2,W^*)\cap \Lambda$ implies that 
$|\mathcal{Q}|_{\Lambda}|\subset\D{P}(H^0(\mathcal{I}_{Z/\Lambda}(2)))$.
We compute then the dimension of $|\mathcal{Q}|_{\Lambda}|$: taking a 9-tuple of points \mbox{$\theta_i,\ i=0, \DM{\ldots,} 8$} spanning $\Lambda$ and choosing linear coordinates $X_0, \DM{\ldots,} X_8$ on $\Lambda$, points in \mbox{$\Lambda\cap {\Gr}(2,W^*)$} are defined by:

$$\sum_{i=0}^{8} X_i \theta_i \in \Lambda\cap \DM{\Gr}(2,W^*)\Longleftrightarrow \wedge^2 (\sum_{i=0}^{8} X_i \theta_i)=0.$$

$\wedge^2 (\sum_{i=0}^{8} X_i \theta_i)$ \DM{gives} an element in $H^0(\mathcal{O}_{\Lambda}(2)) \otimes \bigwedge ^4 W^*\simeq H^0(\mathcal{O}_{\Lambda}(2)) \otimes \bigwedge ^2 W $
that can thus be written as $\sum_{1\le i <j\le 6} Q_{ij} (u_i\wedge u_j)$, with $Q_{ij}\in  H^0(\mathcal{O}_{\Lambda}(2))$. 
$|\mathcal{Q}|_{\Lambda}|$ is the linear system generated by the quadrics $Q_{ij}$.
Writing $\D{P}(B)$ in the form appearing in (\ref{list_plane}), we can write down explicitly generators of $\Lambda$ and compute that $\DM{\dim}(|\mathcal{Q}|_{\Lambda}|)=14$. Hence $|\mathcal{Q}|_{\Lambda}|=\D{P}(H^0(\mathcal{I}_{Z/\Lambda}(2))$) from which we deduce that \mbox{$Z=\Lambda\cap \DM{\Gr}(2,W^*)$}.

	\subsection{Proof of theorem \ref{dim_le3} }
Using the results obtained so far it is now easy to prove Theorem \ref{dim_le3}:
\begin{proof}[Proof of Theorem \ref{dim_le3}]
	\begin{enumerate}
		\item The inequality $n \ge 3$ is simply corollary \ref{unst_le2}. 
		\item From (1), propositions \ref{cor_MM} and \ref{codim_int}, $\Pa\cap \Gr(2,W^*)$ has codimension $3$ hence a general plane $\Pb\subset \Pa$ is disjoint from $\Gr(2,W^*)$.
		\item This follows from proposition \ref{plane_unst}.
		\item From the previous points $\Pa$ always contains a plane of general type $\Pb$ and we saw in last section that the inclusion $\Pa \subset \Pf$ implies that $\Pa\cap \Gr(2,W^*)$ is entirely contained in the scroll $Z$ constructed from $\Pb$. As the stability imposes $\codim(\Pa \cap \Gr(2,W^*))=3$ and as $Z$ is an irreducible variety of dimension 3, we see that we must have $n\le 6$. If ever $n=6$ we would get that \mbox{$\dim(\Pa \cap \Gr(2,W^*))=3$}, so that $\Pa \cap \Gr(2,W^*)=Z$. This is not possible since $\Pa \cap \Gr(2,W^*)$ spans a linear subspace of $\Pa$ but $\langle Z \rangle\simeq \D{P}^8$.
				\end{enumerate}
\end{proof}

\section{Stable linear systems of dimension 3}
In this section we classify 3-dimensional stable linear systems of generic rank 4. We prove the following:
	\begin{thm}\label{dim_3}
		Let $\Pa \subset \Pf$ be a stable three dimensional linear space. Then $\Pa$ is $SL(W)$-equivalent to the space:
		$$\langle e_1\wedge e_4 + e_2\wedge e_5, e_1 \wedge e_6 + e_3\wedge e_5, e_2 \wedge e_6 - e_3 \wedge e_4, (e_1 -e_5) \wedge (e_2 +e_4)\rangle$$
 and $\Pa \cap \Gr(2,W^*)$ consists of two points. 
	\end{thm}
\begin{proof}Let $\Pa$ be a linear space satisfying the hypotheses of the theorem. Applying theorem \ref{dim_le3}, we get that $\Pa\cap \Gr(2,W^*)$ must have dimension zero and that $\Pa$ might be spanned by a plane $\Pb$ of constant rank 4 of general type and by a point $\omega_3\in \Gr(2,W^*)$. Moreover we know that the inclusion $\Pa \subset \Pf$ implies that $\omega_3$ belongs to the scroll $Z\simeq S_{(2,2,2)}$ constructed in the previous section. Keeping the notations formerly adopted, we call $\pi: Z\to \Pb$ the conic fibration over $\Pb$; from the definition of $Z$, there exists then a unique point $\omega \in \Pb$ 
	such that $\omega_3 \in \pi^{-1}(\omega)$. We choose as usual a basis $e_1, \ldots, e_6$ of $W^*$ so that $\Pb=\langle\omega_0,\omega_1,\omega_2\rangle$ with $\omega_0, \omega_1 \omega_2$ as in \ref{list_plane}; up to an appropriate change of coordinates, we can suppose that $\omega_3 \in \pi^{-1}(\omega_0)$ and therefore write $\omega_3$ as:
	$$\omega_3= (t_0 e_1 -t_1 e_5) \wedge (t_0 e_2 + t_1 e_4)= t_0^2(e_1\wedge e_2)+t_0t_1(e_1\wedge e_4+e_2\wedge e_5)+t_1^2(e_4\wedge e_5), \ \hspace{2mm} [t_0:t_1]\in \D{P}^1.$$
	We study now how to choose $[t_0:t_1]\in \D{P}^1$ in order to get stability. From proposition \ref{stab_s1} and remark \ref{2_then_3}, $\Pa$ is not stable if and only if there exists $U\subset W$ of dimension 3 such that $\Gr(2,U)\subset 	\Pa^{\perp}$. Fixing a basis $u_1,\ldots,u_6$ dual to $e_1,\ldots, e_6$, such a space $U$ must then be of the form (prop. \ref{plane_isotropic}) 

	$$U=\DM{\langle}\alpha_3 u_3+ \alpha_6 u_6, \alpha_2 u_2+\alpha_4 u_4, \alpha_1u_1+\alpha_5 u_5\DM{\rangle}$$
	
	with $([\alpha_3:\alpha_6], [\alpha_2:\alpha_4], [\alpha_1:\alpha_5])\in \D{P}^1 \times \D{P}^1 \times \D{P}^1$ , satisfying:
\begin{center}
	\begin{align}\label{quadr-eq}
	\alpha_2\alpha_5-\alpha_4\alpha_1=&0,\\
	\alpha_3\alpha_5-\alpha_6\alpha_1=&0,\\
	\alpha_6\alpha_2+\alpha_3\alpha_4=&0\\
	t_1^2(\alpha_4\alpha_5)+t_0t_1(\alpha_2\alpha_5-\alpha_1\alpha_4)-t_0^2(\alpha_1\alpha_2)=&0;
	\end{align}
\end{center}
	where the last equation arises imposing $\omega_3 \restriction_ U \equiv 0$.
	This system admits solution (a unique to be more precise) if and only if $t_0=0$ or $t_1=0$ namely if and only if $\omega_3$ is one of the two points $\omega \in\mathfrak{C}_{\omega_0}$ such that $\omega_0 \in \D{T}_{\omega}\mathfrak{C}_{\omega_0}$. We can also observe that if ever $t_0=0$ or $t_1=0$, we would get that $\Pa\cap \Gr(2,W^*)$ consists of the point $\omega_3$ counted with multiplicity 2 and this prevents $\Pa$ from being stable as proved in proposition $\ref{codim_int}$.  
	Hence we have that: 
	$$\Pa=\langle \Pb, e_1 - te_5\wedge e_2 + t e_4\rangle, \ t\in \D{C}^*$$
	and it is then clear that $\Pa$ belongs to the $SL(W)$-orbit of the space appearing in the statement of the theorem.  
\end{proof}
\section{Stable linear systems of dimension 4}
We pass now to the classification of stable linear spaces of generic rank 4 and of dimension 4. We prove the following:
	\begin{thm}\label{dim_4}
			Let $\Pa$ be a stable 4-dimensional linear system of \DM{skew-symmetric form}s of generic rank less then or equal to four. Then $\Pa$ is $SL(W)$\DM{-equivalent} to one of the following two spaces:
		$$ \DM{\langle}e_1\wedge e_4 + e_2\wedge e_5, \ e_1 \wedge e_6+ e_3\wedge e_5, \ e_2\wedge e_6 - e_3\wedge e_4, \ e_1\wedge e_2, \ e_4\wedge e_5\DM{\rangle}$$
		or:
		$$ \DM{\langle}e_1\wedge e_4 + e_2\wedge e_5, \ e_1 \wedge e_6+ e_3\wedge e_5, \ e_2\wedge e_6 - e_3\wedge e_4, \ (e_1-e_5)\wedge (e_2+e_4), \ (e_1-e_5)\wedge (e_3+e_6)\DM{\rangle}.$$
		Furthermore in the first case $\Pa$ meets the Grassmannian $\DM{\Gr}(2,W^*)$ along a smooth conic isomorphic to $\D{P}^2\cap \DM{\Gr}(2,4)$; in the second case $\Pa$ intersects $\DM{\Gr}(2,W^*)$ along a pair of disjoint lines.  
		\end{thm}
	\begin{proof}
		Once again we start by writing $\Pa=\langle \Pb, \omega_3, \omega_4\rangle$ with $\Pb=\langle \omega_0,\omega_1,\omega_2\rangle$, a plane of general type, and $ \omega_3, \ \omega_4$ lying on the scroll $Z$ constructed from $\Pb$. Adopting the notations introduced in the former sections, we call $\psi:\Pb \times \D{P}^1\to \Gr(2,W^*)$ the morphism such that $\im(\psi)=Z$ and $\pi:Z\to \Pb$ the conic fibration over $\Pb$. Applying theorem \ref{dim_le3}, we have that $\Pa\cap \Gr(2,W^*)$ has dimension 1; hence we can also suppose that $\omega_3, \ \omega_4$ lie on a (irreducible) curve entirely contained in $Z$. We show that if $\Pa$ satisfies the hypotheses of the theorem, then such a curve is either a conic
		$\mathfrak{C}_{\omega}=\pi^{-1}(\omega)$ for a point $\omega \in \Pb$, or a line $\overline{\omega_3\omega_4}$ entirely contained in $Z$.	
		From the construction of $Z$, each of its points belongs to a unique fiber of $Z \xrightarrow{\pi} \Pb$, hence the following possibilities might 
		occur:
		\begin{itemize}
		\item If $\omega_3, \: \omega_4$ lie on the same fiber of $\pi$, there exists then a point in $\Pb$, $\omega_0$ let's say, such that $\omega_i \in \mathfrak{C}_{\omega_0}, \ i=3,4$. Up to an appropriate change of coordinates on $W$, we might then suppose that the $\omega_i$s $i=0,1,2$ have the form appearing in \ref{list_plane} and that $\omega_3, \ \omega_4$ are the tensors: 
			$$ \omega_3=\DM{e_1 \wedge e_2} , \hspace{3mm} \omega_4=\DM{e_4\wedge e_5}.$$
			In this case $\Pa\cap \Gr(2,W^*)$ consist of the smooth conic $\mathfrak{C}_{\omega_0}$.
			To check the stability of $\Pa$ it is enough to see that for any $\omega\in \mathfrak{C}_{\omega_0}, \ \omega\ne \omega_i, \ i=3,4$, $\langle \Pb,\omega\rangle$
			is stable by theorem \ref{dim_3} and therefore the entire $\Pb$ is stable (as no 3-dimensional subspace of $W$ destabilizes $\langle \Pb,\omega\rangle$, the same holds for $\Pa$).

		\item Suppose now that $\omega_3, \: \omega_4$ don't belong to the same fiber of $\pi$. 
		Without loss of generality we can assume that $\omega_3 \in \mathfrak{C}_{\omega_0}:=\pi^{-1}(\omega_0)$ and $\omega_4 \in \mathfrak{C}_{\omega_1}:= \pi^{-1}(\omega_1)$. 
		For this choice of $\omega_3$ and $\omega_4$, $\Pa\subset \Pf$ holds if and only if $\langle \omega_2, \omega_3, \omega_4 
		\rangle\subset \Pf$. As described in section \ref{sect_hw}, in order to have $\langle \omega_2, \omega_3, \omega_4 
		\rangle\subset \Pf$, the line $\overline{\omega_3\omega_4}$ must be contained in the sextic $D_{\omega_2}= \D{T}_{\omega_2}\Pf\cap Q_{\omega_2}\cap \Pf$ hence we must have $\overline{\omega_3\omega_4}\subset Q_{\omega_2}$. This last inclusion
		 occurs if and only if 
		$\omega_4 \in \D{T}_{\omega_3} Q_{\omega_2}$. 
The conic $\mathfrak{C}_{\omega_1}$ intersects the hyperplane $\D{T}_{\omega_3} Q_{\omega_2}$ in two (possibly coincident) points and these are thus the only two points $\omega$ on $\mathfrak{C}_{\omega_1}$ leading to an inclusion \mbox{$\langle \Pb,\omega_3,\omega\rangle\subset \Pf$.}
We denote by $p_a$ the point on $\D{P}^1$ such that $\omega_3=\psi(p_a, \omega_0)$ and we observe that we have the following equalities:
$$\D{T}_{\omega_3}Q_{\omega_2}\cap \mathfrak{C}_{\omega_1}= \overline{\psi(p_a,\omega_1)\omega_1}\cap \Gr(2,W^*)=\overline{\psi(p_a,\omega_1)\omega_1}\cap \mathfrak{C}_{\omega_1}$$
(to see this just notice that $\overline{\psi(p_a,\omega_1)\omega_1}=\langle \mathfrak{C}_{\omega_1}\rangle \cap \D{T}_{\omega_3}Q_{\omega_2}$);
this intersection consists of a pair of points $\psi(p_a, \omega_1), \ \psi(p_b,\omega_1)$ for $p_b\in \D{P}^1$ ($p_a\equiv p_b$ whenever $\omega_0\in \D{T}_{\omega_3}\mathfrak{C}_{\omega_1}$). We can then suppose $\omega_4=\psi(p_a,\omega_1)$ 
so that, by the construction of $Z$, \mbox{$\overline{\omega_3\omega_4}\subset Z$.} The linear space $\Pa=\langle\Pb,\omega_3, \omega_4\rangle$ will intersect $\Gr(2,W^*)$ along the pair of lines $$l_a=\psi(\{p_a\}\times \overline{\omega_0\omega_1}), \  l_b=\psi(\{p_b\}\times \overline{\omega_0\omega_1})$$
and it will be stable whenever $p_a \ne p_b$ (otherwise the two lines would coincide preventing the stability of $\Pa$ by proposition \ref{codim_int}).
$p_a\ne p_b$ is indeed equivalent to requiring $\omega_0 \notin \D{T}_{\omega_3}\Gr(2,W^*)$ and this condition ensures the stability of $\langle \Pb,\omega_3\rangle$ (by theorem \ref{dim_3}) and consequently of the entire $\Pa$. 
We can then choose the basis $e_1,\dots e_6$ of $W^*$ in such a way that the $\omega_i$s, for $i=0,1,2$ have the form appearing in \ref{list_plane}
and that $\omega_3,  \ \omega_4$ are the tensors: 
			$$\omega_3=(e_1-e_5\DM{) \wedge (}e_2+ e_4) \hspace{3mm} \omega_4= \DM{(e_1-e_5) \wedge (e_3+e_6)}.$$

		\end{itemize}
			\end{proof}
\section{Stable linear systems of dimension 5}
We still need to treat the 5-dimensional case (from theorem \ref{dim_le3} we know indeed that there are no stable linear systems of generic rank 4 and of dimension $\ge 6$).
	We prove the following:
	\begin{thm} Let $\Pa \subset \Pf$ be a stable linear space of dimension 5. Then $\Pa$ is $SL(W)$ equivalent to the space:
	$$ \DM{\langle}e_1\wedge e_4 + e_2\wedge e_5, \ e_1 \wedge e_6+ e_3\wedge e_5, \ e_2\wedge e_6 - e_3\wedge e_4, \ (e_1-e_5)\wedge (e_2+ e_4), (e_1-e_5)\wedge (e_3+e_6), (e_2+e_4)\wedge (e_3+e_6)\DM{\rangle}$$
	and $\Pa$ will meet the Grassmannian $\Gr(2,W^*)$ along a pair of disjoint planes.
	
	\end{thm}
	
	\begin{proof}
	Write $\Pa$ as $\Pa=\langle \Pb, \omega_3, \omega_4, \omega_5\rangle$ with $\Pb=\langle \omega_0,\omega_1,\omega_2\rangle$ a plane of general type and the $\omega_i$s belonging to a divisor of $Z\xrightarrow{\pi}\Pb, \ Z=\psi(\D{P}^1\times \Pb)$, the scroll built from $\Pb$ . 
	As $\Pa\cap \Gr(2,W^*)$ has dimension 2, it can't be contained in a unique fiber of $\pi$ so that we might suppose that $\omega_3, \in \mathfrak{C}_{\omega_0}:=\pi^{-1}(\omega_0), \ \omega_4 \in \mathfrak{C}_{\omega_1}:=\pi^{-1}(\omega_1)$. Adopting the same argument presented for the 4-dimensional case, 
	we  see that from the condition $\Pa \subset \Pf$, we may assume that $\omega_4= \mathfrak{C}_{\omega_1}\cap \D{T}_{\omega_3}\Gr(2,W^*)$ so that the line $l_a:=\overline{\omega_3 \omega_4}$ is entirely contained in $Z$. 
	The intersection of $\langle \Pb, \omega_3,\omega_4\rangle$ with $\Gr(2,W^*)$ will then consist of $l_a$ and of a 
	second line $l_b\subset Z$ and these will have the form:
	$$ l_a=\psi(\{p_a\}\times \overline{\omega_0\omega_1}), \ l_b=\psi(\{p_b\}\times \overline{\omega_0\omega_1})$$
	for a pair of points $p_a, \ p_b$ on $\D{P}^1$.	
	We also saw that $l_a\cup l_b$ locates all the points \mbox{$\omega\in\psi(\D{P}^1 \times \overline{\omega_0 \omega_1})$} for which we have $\langle \Pb,\omega_3, \omega\rangle \subset \Pf$ and $\langle \Pb, \omega_4, \omega\rangle \subset \Pf$. The point $\omega_5$ can not belong to $\psi(\D{P}^1\times \overline{\omega_0\omega_1})$ (otherwise we would get $\Pa \simeq \D{P}^4$) hence we might suppose that $\omega_5 \in \mathfrak{C}_{\omega_2}:= \pi^{-1}(\omega_2)$. Again, from the proof of theorem \ref{dim_4}, we know that there exist at most two points on $\mathfrak{C}_{\omega_2}$ leading to an inclusion 
	$\Pa \subset \Pf$ and we see that the pair defined by: $$\omega_5:=\psi(p_a,\omega_2) \hspace{5mm} \omega_5':=\psi(p_b,\omega_2)$$
	clearly satisfy $\Pa=\langle \Pb, \omega_3, \omega_4, \omega_5\rangle= \langle \Pb, \omega_3, \omega_4, \omega_5'\rangle\subset \Pf$. Such a $\Pa$ is stable whenever $p_a \ne p_b$, this is due to the stability of $\langle \Pb, \omega_3, \ \omega_4\rangle$, (apply theorem \ref{dim_4}) that ensures that no 3-dimensional subspaces of $W$ 
	destabilizes $\Pa$. In this case $\Pa \cap \Gr(2,W^*)$ consists of the two planes $\psi(\{p_a\}\times \Pb)$ and $\psi(\{p_b\}\times \Pb)$. If ever $p_a= p_b$ we would get that $\Pa \cap \Gr(2,W^*)$ consists of a double plane supported on $\psi (\{p_a\}\times \Pb)$, preventing the stability of $\Pa$ 
	by proposition \ref{codim_int}.
	For an appropriate choice of the basis $e_1, \ldots e_6$ we can thus suppose that the $\omega_i$s, $i=0,1,2$ have the form appearing in \ref{list_plane} and that $\omega_3, \ \omega_4, \omega_5$ are the tensors:
	$$ \omega_3= e_1-e_5\wedge e_2+ e_4, \ \omega_4= e_1-e_5\wedge e_3+e_6, \ \omega_5= e_2+e_4\wedge e_3+e_6.$$
	\end{proof}
\section{Appendix: 4-dimensional linear systems of skew-symmetric forms and instanton bundles on cubic threefolds}

We now illustrate one of the main motivation behind our interest in the GIT stability of linear systems of skew-symmetric forms, more precisely of those having dimension 4, 
that is their relation with moduli spaces of instanton bundles on cubic threefolds.

Let's consider then a smooth cubic threefold $X\subset\D{P}^4$. An \textit{instanton bundle} on $X$ is defined as a stable rank 2 vector bundle $\F$ with first Chern class $c_1(\F)=0$ and satisfying the so called \textit{instantonic condition}:
$$H^1(\F(-1))=0.$$
The second Chern class of $\F$ 
is usually referred to as the \textit{charge} of $\F$ and the minimal value of the charge of an instanton on $X$ is $c_2=2$. 
From now on we will only be concerned with instantons of minimal charge that, throughout the rest of the section, will then be simply referred to as instantons. 
The moduli space $\mathcal{M}_X^{in}$ of instantons on a smooth cubic hypersurface $X\subset \D{P}^4$ had been largely studied around the year 2000.
In \cite{MT} and \cite{IM} it was proved the following:
\begin{thm}[\cite{MT} Thm.5.6, \cite{IM} Thm 2.1-Thm 3.2]
There exists a quasi-finite étale morphism of degree 1 $\psi$:
$$\psi: \mathcal{M}_X^{in}\to J(X)$$
inducing an isomorphism of ${\mathcal{M}_X}^{in}$ onto an open subset of the intermediate Jacobian $J(X)$ of $X$.
\end{thm}
Later Druel studied the Gieseker-Maruyama compactification $\mathcal{M}_X(2;0,2,0)$ of ${\mathcal{M}_X}^{in}$ showing that:
\begin{thm}[\cite{Druel} Thm 4.6-Thm 4.8 ]$\mathcal{M}_X(2;0,2,0)$ is a smooth scheme of dimension 5 isomorphic to the blowup of $J(X)$ along the Fano surface $F(X)$ of lines on $X$.
\end{thm}
The ``standard'' Gieseker-Maruyama compactification is obtained by means of semistable torsion free sheaves, the boundary \mbox{$\mathcal{M}_X(2;0,2,0)\setminus {\mathcal{M}_X}^{in}$} is exhaustively described in loc.cit as well.
A sheaf $\F\in \mathcal{M}_X(2;0,2,0)\setminus {\mathcal{M}_X}^{in}$ belongs to one of the following families:
\begin{enumerate}[label=(\roman*)]
\item $\F\simeq \ker( H^0(\mathcal{O}_C(1pt))\otimes {\mathcal{O}_X} \to \mathcal{O}_C(1pt))$ for a smooth conic $C\subset X$. Notice that since the line bundle $\mathcal{O}_C(1pt)$ is generated by its global sections, the evaluation morphism  $H^0(C, \mathcal{O}_{C}(1pt))\otimes \mathcal{O}_X \to \mathcal{O}_{C}(1pt)$ is surjective so that its kernel $\F$ is a torsion free sheaf fitting in a short exact sequence:
	\begin{equation}\label{ses_conic}
	0\longrightarrow \F_{C} \longrightarrow  H^0(\mathcal{O}_C(1pt))\otimes {\mathcal{O}_X}\longrightarrow \mathcal{O}_{C}(1pt)\longrightarrow 0.
	\end{equation}
	A sheaf $\F$ of this kind stable.
	\item $\F\simeq \mathcal{I}_{l_1}\oplus \mathcal{I}_{l_2}$ for a pair of (possibly coincident) lines $l_1,\ l_2$ on $X$. This time $\F$ is semi-stable but clearly not stable.
	
\end{enumerate}

Sheaves of type $(i)$ are parametrized by smooth conics on $X$ whilst sheaves of type $(ii)$ are parametrized by the symmetric square of $F(X)$. 
These families locate, respectively, divisors $\mathcal{B}_X' $ and $\mathcal{B}_X''$ of $\mathcal{M}_X(2;0,2,0)$.

\subsection{ The GIT moduli space of Pfaffian representations of cubic threefolds}
The results that we have just recollected show that we can identify the moduli space ${\mathcal{M}_X}^{in}$ with an open subset of the intermediate Jacobian $J(X)$, but its ``natural'' compactification $\mathcal{M}_X(2;0,2,0)$ is not isomorphic to $J(X)$.
In addition to this, the Gieseker-Maruyama compactification reveals to be troublesome  when applied to threefolds acquiring singularities since in these cases the boundary seems too difficult to treat.
From these issues it is rather natural to ask whether it is possible to find a new compactification of ${\mathcal{M}_X}^{in}$ whose construction can also be extended to singular cubics.
These questions lead us to construct a new moduli space associated to cubic threefolds, the one parameterizing the skew-symmetric presentation maps for the sheaves corresponding to points in $\mathcal{M}_X(2;0,2,0)$. 
\paragraph{Skew-symmetric resolutions of sheaves}
\label{Skew-sym-res}
\

The starting point of our construction is the computation of the minimal resolutions (in $\D{P}^4$) of sheaves $\F$ on $X$ such that $[\F]\in \mathcal{M}_X(2;0,2,0)$.
\begin{itemize}
\item If $[\F]\in {\mathcal{M}_X}^{in}$, it was observed in Beauville that $\F(1)$ is a rank 2 Ulrich bundle on $X$ so that $\F$ has the following resolution:
\begin{equation}\label{ses-prel}
0\longrightarrow {\mathcal{O}_{\D{P}^4}(-2)}^{\oplus 6} \overset{M}\longrightarrow \mathcal{O}_{\D{P}^4}(-1)^{\oplus 6}\overset{N}\longrightarrow \F\longrightarrow 0
\end{equation}
where $M$ is a $6\times 6$ skew-symmetric matrix whose entries are linear forms on $\D{P}^4$. This also means that, since $X=\Supp(\F)=\Supp(\coker(M))$, $M$ provides a \textit{Pfaffian representation} of $X$ that is, $X$ is defined by the equation $\Pf(M)=0$. 
Notice that conversely, if we are given $M$, a skew-symmetric matrix of size 6 whose entries are linear forms and having generic rank 6, we get a short exact sequence of $\mathcal{O}_{\D{P}^4}$-modules:

$$0\longrightarrow {\mathcal{O}_{\D{P}^4}(-1)}^{\oplus 6}\overset{M}\longrightarrow {\mathcal{O}_{\D{P}^4}}^{\oplus 6}\longrightarrow \coker(M) \longrightarrow 0;$$

$\coker(M)$ is a sheaf supported on the cubic $X$ defined by the equation $\Pf(M)=0$ and computing its cohomology we deduce that its restriction to the smooth locus of $X$ is a vector bundle.  
In particular, if $X$ is non-singular, it is easy to show that $\coker(M)\otimes \mathcal{O}_{\D{P}^4}(-1)$ is an instanton bundle. 
\item If $[\F]\in \mathcal{M}_X(2;0,2,0)\setminus {\mathcal{M}_X}^{in}$ we can easily compute,  from the resolution of $\mathcal{O}_C(1pt)$ and \ref{ses_conic} if ever $[\F]\in \mathcal{B}_X'$ or from the resolutions of ideal sheaves of lines if $[\F]\in \mathcal{B}_X''$ instead, 
that the minimal resolution of $\F$ in $\D{P}^4$ has the following form:
\begin{equation}\label{res-boundary}
0\xrightarrow{\ }\OOO(-3)^{\oplus 2}\rightarrow
\OOO(-3)^{\oplus 2}\oplus \OOO(-2)^{\oplus 6}\xrightarrow{B}\OOO(-1)^{\oplus 6}
\xrightarrow{\ }\F\xrightarrow{\ }0,
\end{equation}
in which $B=(\beta'|\beta)$ is a 6-by-8 matrix obtained by concatenation of a $6\times 2$ matrix of quadratic forms $\beta'$ with a $6\times 6$ skew-symmetric matrix of linear forms $\beta$ such that $\Pf(\beta)=0$.
Whenever $[\F]\in \mathcal{B}_X'$ is the sheaf associated to a conic \mbox{$C\subset X$}, choosing coordinates $X_0,\ldots, X_4$ on $\D{P}^4$ in such a way that $C$ has equations \mbox{$X_4=X_3=X_1^2-X_0X_2=0$}, the matrix $\beta$ can be reduced to the form:
$$
\beta = \begin{pmatrix}0&
0&
0&
{X}_{4}&
0&
{-{X}_{3}}\\
0&
0&
{-{X}_{4}}&
0&
{X}_{3}&
0\\
0&
{X}_{4}&
0&
0&
{X}_{2}&
{-{X}_{1}}\\
{-{X}_{4}}&
0&
0&
0&
{-{X}_{1}}&
{X}_{0}\\
0&
{-{X}_{3}}&
{-{X}_{2}}&
{X}_{1}&
0&
0\\
{X}_{3}&
0&
{X}_{1}&
{-{X}_{0}}&
0&
0\\
\end{pmatrix}
$$
The matrix $\beta$ defines a 4-dimensional linear system of skew-symmetric forms having generic rank 4 and having rank 2 along the conic $C$.

If $[\F]\in \mathcal{B}_X''$ is the point defined by the sheaf $\F\simeq \mathcal{I}_{l_1}\oplus \mathcal{I}_{l_2}$ and we choose coordinates such that 
$l_1$ is defined by the equations $\{X_0=0, \ X_1=0,\ X_2=0\}$ and $l_2$ is defined by  $\{X_2=0, \ X_3=0,\ X_4=0\}$
we have:

\begin{equation*}\label{matrix-lines}
\beta=\begin{pmatrix}
0&X_2&-X_1&0&0&0\\
-X_2&0&X_0&0&0&0\\
X_1&-X_0&0&0&0&0\\
0&0&0&0&X_4&-X_3\\
0&0&0&-X_4&0&X_2\\
0&0&0&X_3&-X_2&0\\
\end{pmatrix}
\end{equation*}

This time $\beta$ defines a 4-dimensional linear system of generic rank 4 that has rank 2 along $l_1 \cup l_2$.

By theorem \ref{dim_4}, we can therefore see that in each case, the matrix $\beta$ belongs to one of the stable $SL(6,\D{C})$-orbits of linear systems of skew-symmetric forms of generic rank 4.
\end{itemize}

\paragraph{The GIT moduli space of Pfaffian representations} 
\

Keeping the notation formerly adopted, we denote by $W$ and $V_5$ two complex vector spaces of dimension 6 and 5 respectively. We set:  $$\mathcal{P}:=\D{P}({V_5}^*\otimes \bigwedge^2 W^*)\simeq \D{P}^{74}$$ the projective space of $6\times 6$ skew-symmetric matrices whose entries are elements in ${V_5}^*$. We consider $\mathcal{P}^{in}$ the open parameterizing matrices $M \in \mathcal{P}$ such that the cubic defined by the equation $\Pf(M)=0$ is smooth.   Let $\mathfrak{M}^{in}$ be the moduli space of torsion sheaves on $\D{P}^4$ with supports on smooth cubic hypersurfaces $X\subset \D{P}^4$, whose restrictions to $X$ are instanton bundles:
$$\mathfrak{M}^{in}:=\{[\F] \mid \F \ \text{is an instanton on a smooth cubic} \ X\subset \D{P}^4\};$$
define $\mathfrak{M}$ as the closure of $\mathfrak{M}^{in}$ in the moduli space of sheaves on $\D{P}^4$. We call \mbox{$\mathcal{U}\subset |\mathcal{O}_{\D{P}^4}(3)|$} 
the open subset of $|\mathcal{O}_{\D{P}^4}(3)|$ parameterizing smooth cubics. As for what was previously discussed, we deduce that we have a commutative diagram:
$$\xymatrix{
	\mathcal{P}^{in} \ar[r]^{\tau} \ar[dr]_{\Pf} & \mathfrak{M}^{in} \ar[d]^{\rho} \\
	& |\mathcal{O}_{\D{P}^4}(3)|
}
$$ 

where $M\mapstoo {\tau} \coker (M)\otimes \mathcal{O}_{\D{P}^4}(-1)\mapstoo{\rho} \Supp(\coker (M))=\{\Pf(M)=0\}$. 
According to \cite{Bea}, the morphism $\tau:\mathcal{P}^{in}\to \mathfrak{M}^{in}$ is a principal bundle with structure group $PGL(W)$;
the group $GL(W)$ acts on the resolution \ref{ses-prel} of an instanton $\F$ by:
$$g\cdot(M,N)=(gM g^T, Ng^{-1}), \ g \in GL(W).$$

Since our aim is to construct a new moduli space related to instanton bundles on a cubic threefold $X$ and as taking free resolutions in $\D{P}^4$, we can associate  elements in $\mathcal{P}$ to each point in $\mathcal{M}_X(2;0,2,0)$,
we study the moduli space $\mathfrak{P}$ of Pfaffian representations of cubic threefolds. The group $GL(W)$ acts on $\mathcal{P}$ by conjugation, 
therefore $\mathfrak{P}$ can be obtained by means of Geometric Invariant Theory (GIT) as the GIT quotient:
$$\mathfrak{P}:=\mathcal{P}^{ss}\git SL(W)$$
where $\mathcal{P}^{ss}$ is the open parameterizing semistable matrices. 
Note that from \ref{pfaff_ss} and \ref{dim_4}, 
every presentation map of sheaves $[\F]\in \mathcal{M}_X(2;0,2,0)$, with $ X\in \mathcal{U}$ is defined by a semistable point in $\mathcal{P}$ and locates therefore a point in the moduli $\mathfrak{P}$. 
By the construction of the GIT quotient (\cite{GIT}, \cite{Dol}), we get the existence of an open $\mathfrak{P}^s\subset \mathfrak{P}$ that is a geometric quotient for the $SL(W)$ action on $\mathcal{P}^s$ (the set of stable points); this together with the fact that $\mathcal{P}$ is a linear space of dimension 74, implies that $\mathfrak{P}$ is a 39-dimensional irreducible compact projective scheme. 
$\Pf$ and $\tau$ extends to rational maps $\tau: \mathcal{P}\dashrightarrow \mathfrak{M}$ and $\Pf: \mathcal{P}\dashrightarrow |\mathcal{O}_{\D{P}^4}(3)|$, both of them are $GL(W)$ invariant and induce thus a commutative diagram:
\begin{equation}\label{triangle_diag}
\xymatrix{
	\mathfrak{P} \ar@{-->}[r]^{\overline{\tau}} \ar@{-->}[dr]_{\overline{\Pf}} & \mathfrak{M} \ar[d]^{\rho} \\
	& |\mathcal{O}_{\D{P}^4}(3)|.
}
\end{equation}
$\overline{Pf}$ is a rational map whose generic fiber is compact and of dimension 5, this means in particular that for $X\in \mathcal{U}, \ \overline{Pf}^{-1}(X)$ is a new compactification of $\mathcal{M}_X^{in}$ different from the ``standard'' Gieseker-Maruyama one.
The map $\overline{\tau}$ is a birational morphism (inducing an isomorphism between the quotient of $\mathcal{P}^{in}\cap \mathcal{P}^s$ and $\mathfrak{M}^{in}$) and it is our objective to get a detailed description of the locus $\mathcal{S}:=\{ \beta \in \mathcal{P} \mid \Pf(\beta)=0\}$ at neighborhoods of points in $\mathcal{S}\cap \mathcal{P}^s$; this could indeed help us to understand the behavior of $\overline{\tau}$ at the boundary of $\mathfrak{M}$.

\end{document}